\documentclass[11pt,reqno,a4paper]{amsart}
\usepackage{amsfonts}
\usepackage{float}
\usepackage{amsmath,amssymb}
\usepackage{thmtools}
\usepackage{thm-restate}
\usepackage{lscape}
\usepackage[dvipdfmx]{graphicx, color}
\usepackage[all]{xy}

\usepackage[
  top=30mm,
  bottom=30mm,
  left=30mm,
  right=30mm
]{geometry}

\def\qed{\hfill $\Box$}
\def\proof{\noindent {\sl Proof} :\;  }

\newcommand{\A}{\mathcal{A}}
\newcommand{\B}{\mathcal{B}}

\newcommand{\D}{\mathcal{D}}
\newcommand{\E}{\mathcal{E}}
\newcommand{\F}{\mathcal{F}}

\newcommand{\R}{\mathbb{R}}

\newcommand\rank{\mbox{\rm rank}\,}

\newcommand\Imag{\mbox{\rm Im} }

\newcommand\tr{\mbox{\rm tr} }

\def\qed{\unskip\nobreak\hfill$\Box$}
\def\proof{\noindent {\sl Proof}.\;  }

\def\rd{\partial}

\def\bx{\mbox{\boldmath $x$}}

\newcommand{\til}[1]{\widetilde{#1}}
\newcommand{\ha}[1]{\widehat{#1}}
\newcommand{\barr}[1]{\overline{#1}}

\allowdisplaybreaks

\numberwithin{equation}{section}

\theoremstyle{definition}

\newtheorem{thm}{Theorem}[section]

\newtheorem{lem}[thm]{Lemma}
\newtheorem{prop}[thm]{Proposition}
\newtheorem{dfn}[thm]{Definition}

\newtheorem{rem}[thm]{Remark}
\newtheorem{exam}[thm]{Example}

\begin{document}
\title[Equiaffine immersion, projective flatness and quasi-Codazzi structure]
{Equiaffine immersion, projective flatness and quasi-Codazzi structure}
\author[K.~Kayo]{Kaito Kayo}
\address[K.~Kayo]{D2, Graduate School of Science, Hokkaido University,
Sapporo 060-0810, Japan}
\email{kayo.kaito.b9@elms.hokudai.ac.jp}
\date{\today}
%

%
%
\dedicatory{}
\begin{center}
\end{center}
\begin{abstract}

In the present paper, we study an extended theory of statistical manifolds in application to affine differential geometry.
Any smooth hypersurface $M \subset \mathbb{R}^{n+1}$ with a transverse vector field $\xi$ naturally admits a symmetric $(0, 2)$-tensor $h$ and a torsion-free connection $\nabla$ on $M$ so that $\nabla h$ is totally symmetric.
Here $h$ may be degenerate (i.e., not a pseudo-Riemannian metric) in general.
As a generalization of classical theorem due to Weyl, Radon, Nomizu, Kurose and others, we show, roughly saying, that $M$ with $\xi$ is {\em equiaffine} if and only if $(h, \nabla)$ defines a {\em quasi-Codazzi structure}, previously introduced by the author, and it admits a {\em projectively flat} dual connection with symmetric Ricci contraction.
This is a direct consequence from our quasi-Codazzi theory, which is built in a more general context as a submanifold theory in para-Hermitian geometry.

\end{abstract}

\maketitle

{\normalsize
\setcounter{tocdepth}{2}
}


\section{Introduction}
\label{Intro}

The theory of statistical manifolds has been well examined within the framework of affine differential geometry (see e.g., \cite{Matsuzoe, Nomizu-Sasaki}).
To begin with, we pick up a typical topic on affine immersions.

Any immersion $f : M \rightarrow \R^{n+1}$ from an $n$-dimensional manifold $M$ equipped with a transversal vector field $\xi$ along $f$ naturally admits a symmetric $(0, 2)$-tensor field $h$ and a torsion-free affine connection $\nabla$ on $M$ such that $C := \nabla h$ is a totally symmetric cubic tensor.
We call $h$ the {\em second fundamental form} and $\nabla$ the {\em induced connection}.
Note that $h$ may be degenerate (i.e., $h$ is not a pseudo-Riemannian metric).
Conversely, given a torsion-free affine connection $\nabla$ and a symmetric $(0, 2)$-tensor field $h$ on $M$ with totally symmetric $\nabla h$, it is an important and interesting problem to determine the conditions under which there exist a pair $\{f, \xi\}$ of immersion $f : M \rightarrow \R^{n+1}$ and transversal vector field $\xi$ inducing $\nabla$ and $h$.
In particular, in case that $h$ is non-degenerate, this structure $(h,\nabla)$ is called a statistical structure (or a Codazzi structure), and the corresponding problem has been extensively studied for a long time.
In 1918, Radon investigated necessary and sufficient conditions for a two-dimensional manifold to be realized as an equiaffine immersion~\cite{Radon}. Later, in 1990, Dillen--Nomizu--Vranken~\cite{DNV} extended it to higher dimensions by incorporating a dual connection $\nabla^{*}$ (we will precisely state their result as Theorem~\ref{thm:nondegenerate and equiaffine <=> Codazzi structure} later).

Our aim of the present paper is to generalize this classical theorem to the setting where the non-degeneracy condition of $h$ is dropped.
That is, we extract equiaffineness of $\{f, \xi\}$ solely and characterize it as third order geometry. In our previous paper \cite{Kayo}, we have introduced a novel concept of a {\em quasi-Codazzi structure} generalizing the Codazzi structure by allowing the underlying metric $h$ to be degenerate. A main point of the present paper is to extend the notion of {\em projective flatness} (\cite{Eisenhart, Kobayashi-Nomizu, Weyl}) to our quasi-Codazzi setting, and then, by using this, we show that $\{f, \xi\}$ is an equiaffine immersion if and only if it induces a quasi-Codazzi structure with projectively flat connection $\nabla^{-}$ (an alternative to $\nabla^{*}$) and symmetric Ricci contraction, provided $M$ is simply connected.
See Theorems~\ref{main_thm_1} and \ref{main_thm_2} for precise statements.

The author's quasi-Codazzi theory \cite{Kayo} is about {\em geometry of the cubic tensor $C$ regardless of the existence of the pseudo-Riemannian metric $h$} and it fits the theory of weak contrast functions in the sense of Eguchi and Matumoto~\cite{Eguchi, Matumoto}.
Due to the degeneration of $h$, the connection $\nabla$ and its dual $\nabla^{*}$ may no longer exist, so some new innovation is needed to provide substitutes.
In our theory, those are affine connections $\nabla^{+}$ and $\nabla^{-}$ of {\em coherent tangent bundles} $E^{+}$ and $E^{-}$, respectively (substitutes for the (co)tangent bundle $TM$ and $T^{*}M$).
Then a symplectic vector bundle $E = E^{+} \oplus E^{-}$ equipped with a para-Hermitian structure (substitute for $T(T^{*}M) = TM \oplus T^{*}M$) naturally appears in our definition of quasi-Codazzi structure (Definition~\ref{dfn:quasi-Codazzi structure}).
Our quasi-Codazzi theory would be a relevant generalization of statistical manifolds with broad applications and promising potential; indeed, it should be understood as a sort of {\em differential geometry of Lagrange subbundles in para-Hermitian symplectic vector bundles or submanifold theory in para-Hermitian geometry} (Example~\ref{exam:symple_example}).

The rest of the present paper is organized as follows.
In Section~\ref{sec:2}, we review the definition and properties of quasi-Codazzi structures.
In Section~\ref{sec:3}, we generalize the notion of projective flatness to coherent tangent bundles and investigate the necessary and sufficient conditions for projective flatness.
In Section~\ref{sec:4}, we show that an equiaffine immersion induces a quasi-Codazzi structure.
Moreover, we prove that, conversely, given a quasi-Codazzi structure, projective flatness is involved in the conditions for the existence of an equiaffine immersion inducing the structure.

In this paper, we let $M$ be an $n$-dimensional manifold.
For a connection $\nabla$ on a vector bundle, we denote its curvature tensor field by $R^{\nabla}$, and we use the same notation for connections on the tangent bundle $TM$.
In the case where $\nabla$ is an affine connection on the tangent bundle $TM$, we denote its Ricci tensor field by $Ric^{\nabla}$.
We adopt the Einstein summation convention.
Unless otherwise stated, all geometric objects considered in this paper are assumed to be of class $C^{\infty}$.

\vskip\baselineskip
\vskip\baselineskip
\vskip\baselineskip
\section{Quasi-Codazzi structure}
\label{sec:2}


First, we recall the definition of Codazzi structure.

\begin{dfn}
  [cf., \cite{Lauritzen, Matsuzoe}]
  Let $h$ be a pseudo-Riemannian metric and $\nabla$ a connection on $M$.
  A dual connection $\nabla^{*}$ of $\nabla$ with respect to $h$ is defined by
    \begin{align}
      Xh(Y,Z) = h(\nabla_{X}Y,Z) + h(Y,\nabla^{*}_{X}Z)
      \label{def of dual connection}
    \end{align}
  for $X,Y,Z \in \Gamma(TM)$.
  We call a triplet $(h,\nabla,\nabla^{*})$ a {\em statistical structure} or a {\em Codazzi structure} on $M$ if $\nabla$ and $\nabla^{*}$ are torsion-free.
  Since $h$ is nondegenerate, a dual connection exists uniquely.
  In a special case where $\nabla$ and $\nabla^{*}$ are flat, we call it a {\em dually flat} or {\em Hessian structure}.
\end{dfn}

If the metric is degenerate, this geometric structure is not applicable, as the dual connection does not exist (\cite[Proposition~2.5]{Kayo}).
In this section, we briefly review quasi-Codazzi structures introduced in \cite{Kayo} and their properties.

\begin{dfn}
  [\cite{SUY-F1}]
  Let $\E$ be a rank $n$ vector bundle over $M$, $\nabla^{\E}$ a connection on $\E$, and $\phi:TM \rightarrow \E$ a bundle map over $M$.
  We say that $\nabla^{\E}$ is {\em relatively torsion-free} with respect to $\phi$ if
  \begin{align}
  \nabla^{\E}_{X}\phi(Y) - \nabla^{\E}_{Y}\phi(X) - \phi([X,Y])= 0
  \label{dfn:relatively torsion-free}
  \end{align}
  holds for any $X,Y \in \Gamma(TM)$.
  Moreover, the triplet $(\E,\phi,\nabla^{\E})$ is called a {\em coherent tangent bundle} if $\nabla^{\E}$ is relatively torsion-free with respect to $\phi$.
\end{dfn}

The notion of coherent tangent bundle in \cite{SUY-F1} is considered for the triplet $(\E,\phi,\nabla^{\E})$ with a metric of $\E$.
Note that a bundle map means a fiberwise linear map, i.e., $\phi_{p}:T_{p}M \to \E_{p}$ is linear for every $p \in M$, so the rank may vary according to points $p$.
If it is isomorphic fiberwise, we call it a bundle isomorphism.

\begin{dfn}[\cite{LS05}]
  \label{dfn:para-Hermitian vector bundle}
  Let $E$ be a rank $2n$ vector bundle on $M$ and $I:E\rightarrow E$ a bundle isomorphism.
  Suppose that $\theta \in \Gamma (E^{*} \otimes E^{*})$ is symmetric and nondegenerate.
  We call the triplet $(E,\theta,I)$ a {\em para-Hermitian vector bundle} on $M$ if the following conditions hold:
  \begin{itemize}
    \item[(i)]
    $I^{2}=id$ and $I \neq id$;
    \item[(ii)]
    The eigenvalues of $I$ are $\pm 1$.
    Let $E^{+}$ (resp. $E^{-}$) be the subbundle of $E$ consisting of eigenvectors with the eigenvalue $1$ (resp. $-1$).
    Then $E^{+}$ and $E^{-}$ have the same rank and  $E=E^{+} \oplus E^{-}$ holds;
    \item[(iii)]
    For $\eta,\zeta\in\Gamma(E)$, $\theta(I\eta,\zeta)+\theta(\eta,I\zeta)=0$ holds.
  \end{itemize}
\end{dfn}

\begin{dfn}[\cite{Kayo}]
  \label{dfn:dual connection}
  Let $(E,\theta, I)$ be a para-Hermitian vector bundle on $M$.
  Set a connection $\nabla^{+}$ on $E^{+}$ and a connection $\nabla^{-}$ on $E^{-}$.
  These connections are called {\em mutually dual} if, for $X\in \Gamma(TM)$, $\eta^{+}\in \Gamma(E^{+})$, and $\eta^{-}\in \Gamma(E^{-})$, it holds that
  \[
  X\theta(\eta^{+},\eta^{-}) = \theta(\nabla^{+}_{X}\eta^{+},\eta^{-}) +  \theta(\eta^{+},\nabla^{-}_{X}\eta^{-}).
  \]
  The connection $\nabla^{+}$ (resp. $\nabla^{-}$) is called the {\em dual} of $\nabla^{-}$ (resp. $\nabla^{+}$).
\end{dfn}

A para-Hermitian vector bundle $(E,\theta, I)$ canonically admits a symplectic form $\omega$ on $E$ defined by
\begin{align}
  \label{dfn:symplectic_form}
  \omega(\eta,\zeta):=\theta(\eta,I\zeta)
\end{align}
for $\eta,\zeta \in \Gamma(E)$.
We consider a bundle map $\Phi:TM \rightarrow E$ such that, for every $p\in M$, $\Imag \Phi_{p} =\Phi_{p}(T_{p}M)$ is a Lagrange subspace of $E_{p}$, i.e., $\omega$ vanishes on $\Imag \Phi_{p}$ and $\rank\Phi_{p} = n$ holds.
For this bundle map $\Phi$, we have two bundle maps $\Phi^{+}:=\pi^{+}\circ\Phi$ and $\Phi^{-}:=\pi^{-}\circ\Phi$, where $\pi^{+}:E\rightarrow E^{+}$ and $\pi^{-}:E\rightarrow E^{-}$ are natural projections.
Then we can write $\Phi = \Phi^{+}\oplus\Phi^{-}$.
Accordingly, the pullback $h:=\Phi^{*}\theta \in \Gamma (T^{*}M \otimes T^{*}M)$ is given by
\begin{align}
h(X,Y) = \theta(\Phi(X),\Phi(Y)) = 2\theta(\Phi^{+}(X),\Phi^{-}(Y)) \label{def:degenerate metric}
\end{align}
for $X,Y\in \Gamma(TM)$.

\begin{dfn}[\cite{Kayo}]
  \label{dfn:quasi-Codazzi structure}
  We call $(h,(E,\theta,I),\Phi,\nabla^{+},\nabla^{-})$ a {\em quasi-Codazzi structure} on $M$ if $\nabla^{+}$ and $\nabla^{-}$ are relatively torsion-free with respect to $\Phi^{+}$ and $\Phi^{-}$, respectively.

  In particular, a quasi-Codazzi structure $(h,(E,\theta,I),\Phi,\nabla^{+},\nabla^{-})$ is a {\em quasi-Hessian structure} introduced by Nakajima--Ohmoto~\cite{Nakajima-Ohmoto} if and only if $\nabla^{+}$ and $\nabla^{-}$ are mutually flat connections.
\end{dfn}

\begin{exam}
  [\cite{Kayo}]
  \label{exam:symple_example}
  We give a simple example.
  Let $N$ be a $2n$-dimensional manifold equipped with a para-Hermitian vector bundle $(TN = E^{+}\oplus E^{-},\theta, I)$, and let $\nabla = \nabla^{+}\oplus \nabla^{-}$ be an invariant torsion-free connection with respect to $I$.
  Then $(N,\omega)$ is a symplectic manifold, where $\omega$ is defined by \eqref{dfn:symplectic_form}.
  If $M \subset N$ is a Lagrangian submanifold, we obtain a quasi-Codazzi structure on $M$.
  Indeed, we set $\Phi := \iota_{*}$ and $h := \Phi^{*}\theta$, where $\iota:M \rightarrow N$ is the inclusion map.
  Then $(h,(TN,\theta,I),\Phi, \nabla^{+},\nabla^{-})$ is a quasi-Codazzi structure on $M$.
\end{exam}

\begin{prop}
  [\cite{Kayo}]
  \label{prop of quasi-Codazzi str}
  Let $(E,\theta, I)$ be a para-Hermitian vector bundle on $M$.
  Set a connection $\nabla^{+}$ on $E^{+}$ and a connection $\nabla^{-}$ on $E^{-}$.
  We define $C \in \Gamma(T^{*}M \otimes T^{*}M \otimes T^{*}M)$ by
  \begin{align}
    C(X,Y,Z) = &-2\{ \theta(\nabla^{+}_{X}\Phi^{+}(Y),\Phi^{-}(Z)) - \theta(\Phi^{+}(Z), \nabla^{-}_{X}\Phi^{-}(Y)) \}
    \label{dfn:cubic tensor of quasi-Codazzi}
  \end{align}
  for $X,Y,Z\in \Gamma(TM)$.
  If $\Phi^{+}$ is a bundle isomorphism and $\nabla^{+}$ is relatively torsion-free with respect to $\Phi^{+}$, then $C$ coincides with $\nabla h$, where $\nabla$ is a torsion-free connection on $M$ defined by $\Phi^{+}(\nabla_{X}Y) = \nabla^{+}_{X}\Phi^{+}(Y)$, and the following are equivalent:
  \begin{itemize}
    \item[(i)]
      $\nabla^{-}$ is relatively torsion-free with respect to $\Phi^{-}$;
    \item[(ii)]
      $C$ is totally symmetric.
  \end{itemize}
\end{prop}

\begin{dfn}[\cite{Kayo}]\label{def of isomorphic}
  Let $(h_{i},(E_{i}=E_{i}^{+}\oplus E_{i}^{-}, \theta_{i}, I_{i}), \Phi_{i}=\Phi_{i}^{+}\oplus \Phi_{i}^{-}, \nabla_{i}^{+}, \nabla_{i}^{-} )$ be quasi-Codazzi structures for $i=1,2$.
  These structures are {\it isomorphic} if there exists a bundle isomorphism $\F:E_{1} \rightarrow E_{2}$ on $M$ that satisfies
  \begin{itemize}
    \item [(i)]
    $\F \circ \Phi_{1} = \Phi_{2}$;
    \item[(ii)]
    $\F \circ I_{1} = I_{2}  \circ \F$;
    \item[(iii)]
    $\theta_{1}(\eta,\zeta) = \theta_{2}(\F(\eta),\F(\zeta))$ for $\eta,\zeta \in \Gamma(E_{1})$;
    \item[(iv)]
    $\F(\nabla^{+}_{1\;X}\eta^{+}) = \nabla^{+}_{2\;X}\F(\eta^{+})$ for $\eta^{+}\in \Gamma(E_{1}^{+})$, and $\F(\nabla^{-}_{1\;X}\eta^{-}) = \nabla^{-}_{2\;X}\F(\eta^{-})$ for $\eta^{-}\in \Gamma(E_{1}^{-})$.
  \end{itemize}
\end{dfn}

Let $(h,(E,\theta,I),\Phi,\nabla^{+},\nabla^{-})$ be a quasi-Codazzi structure.
If $h$ is nondegenerate, $\Phi^{+}$ and $\Phi^{-}$ are isomorphisms by \eqref{def:degenerate metric}.
Hence, we can define two affine connections $\nabla$ and $\nabla^{*}$ on $M$ by
\[
\Phi^{+}(\nabla_{X}Y) = \nabla^{+}_{X}\Phi^{+}(Y), \quad \Phi^{-}(\nabla^{*}_{X}Y) = \nabla^{-}_{X}\Phi^{-}(Y)
\]
for $X,Y\in\Gamma(TM)$.
We can easily check that these connections are torsion-free and satisfy \eqref{def of dual connection}.
Then the triplet $(h,\nabla,\nabla^{*})$ is a Codazzi structure.
A Codazzi structure can be regarded as a quasi-Codazzi structure with non-degenerate $h$ in a canonical way, see \cite[Example 3.22]{Kayo}.

We introduce the definition of weak contrast functions and their relation to quasi-Codazzi structures.
For details, see \cite{Eguchi, Kayo}.
Let $\rho:U \rightarrow \R$ be a function defined on an open neighborhood $U$ of the diagonal $\Delta_{M} \subset M\times M$.
A function $\rho[X_{1} \dots X_{k}|Y_{1} \dots Y_{l}]$ is defined by
\[
(X_{1})_{p} \dots (X_{k})_{p} (Y_{1})_{q} \dots (Y_{l})_{q} (\rho(p,q)) \big|_{p=q=r}
\]
for every $r \in M$ and $X_{i},Y_{j}\in \Gamma(TM)$ $(1 \leq i \leq k, 1 \leq j \leq l)$.
We also define $\rho[X|-](r) = X_{p}\rho(p,q)\big|_{p=q=r}$ and $\rho[-|X](r) = X_{q}\rho(p,q)\big|_{p=q=r}$ for $X\in\Gamma(TM)$.
We call $\rho$ a {\it contrast function} if, for $X,Y\in\Gamma(TM)$, it satisfies the following conditions:
\[
\text{(i)}\ \rho[-|-] = \rho(p,p) = 0; \qquad
\text{(ii)}\ \rho[X|-] = \rho[-|X]=0;
\]
\begin{center}
  (iii) A symmetric $(0,2)$-tensor field $h(X,Y):=-\rho[X|Y]$ on $M$ is nondegenerate.
\end{center}
If it satisfies only (i) and (ii), $\rho$ is called a {\it weak contrast function}.

It is well known that a contrast function uniquely induces a Codazzi structure (see \cite{Eguchi}).
The following theorem shows the converse.

\begin{thm}
  [\cite{Matumoto}]
  \label{thm:weak contrast function}
  Given a symmetric $(0,2)$-tensor field $h$, i.e., a possibly degenerate metric, and a totally symmetric cubic tensor field $C$, we get a weak contrast function $\rho$ such that
  \begin{align*}
    h(X,Y) &=-\rho[X|Y] \; (=\rho[XY|-] = \rho[-|XY]), \\
    C(X,Y,Z) &=-\rho[Z|XY] + \rho[XY|Z]
  \end{align*}
  holds for $X,Y,Z\in\Gamma(TM)$.
  In particular, each quasi-Codazzi structure (resp. Codazzi structure) admits a weak contrast function (resp. a contrast function).
\end{thm}

\vskip\baselineskip
\vskip\baselineskip
\vskip\baselineskip
\section{Projective flatness on coherent tangent bundle}
\label{sec:3}

\subsection{Projectively flat connection}
\label{sec:3-1}

\begin{dfn}
  [\cite{Eisenhart, Nomizu-Sasaki}]
  We say that two torsion-free affine connections $\nabla$ and $\widehat{\nabla}$ on $M$ are {\em projectively equivalent} if there exists a $1$-form $\rho$ such that
  \[
  \widehat{\nabla}_{X}Y = \nabla_{X}Y + \rho(X)Y + \rho(Y)X
  \]
  holds for $X,Y\in\Gamma(TM)$.
  In particular, $\nabla$ is said to be {\it projectively flat} if $\nabla$ is projectively equivalent to some flat connection in a neighborhood of an arbitrary point of $M$.
\end{dfn}

Let $c:I\rightarrow M$ be a curve, where $I(\neq \emptyset) \subset \R$ is an open interval, and set $\dot{c}:=\frac{d}{dt}c$.
We say $c$ is a {\em $\nabla$-pre-geodesic} if $\nabla_{\dot{c}}\dot{c} = f\dot{c}$ for some function $f$.
If $\nabla$ and $\widehat{\nabla}$ are projectively equivalent, then these two connections have the same pre-geodesics \cite{Eisenhart, Nomizu-Sasaki}.

The above definition can be reformulated in terms of curvature as follows.

\begin{thm}
  [\cite{Eisenhart, Nomizu-Sasaki, Weyl}]
  \label{thm:condition of projectively flat}
    Let $\nabla$ be a torsion-free connection with symmetric Ricci tensor field $Ric^{\nabla}$.
    Then the following conditions are equivalent:
    \begin{itemize}
        \item[(i)]
            the connection $\nabla$ is projectively flat;
        \item[(ii)]
            if $\dim M =2$, $\nabla Ric^{\nabla}$ is totally symmetric, and if $\dim M = n \geq 3$, it holds that for $X,Y,Z\in\Gamma(TM)$,
            \[
            R^{\nabla}(X,Y)Z = \frac{1}{n-1} \{ Ric^{\nabla}(Y,Z)X - Ric^{\nabla}(X,Z)Y \}.
            \]
    \end{itemize}
\end{thm}

This is a classical theorem shown by Weyl~\cite{Weyl} a hundred years ago, see also Eisenhart~\cite{Eisenhart}, Nomizu--Sasaki~\cite[Theorem 3.3]{Nomizu-Sasaki} and Belgun~\cite[Theorem 1]{Belgun}.

\begin{rem}
  [\cite{Eisenhart}]
  In the case of $\dim M =2$, $R^{\nabla}(X,Y)Z = Ric^{\nabla}(Y,Z)X - Ric^{\nabla}(X,Z)Y$ holds by the definition of the Ricci tensor field.
  In the case of $\dim M \geq 3$, $\nabla Ric^{\nabla}$ is automatically totally symmetric by the second Bianchi identity.
\end{rem}

The following proposition follows directly from the first Bianchi identity \cite{Kobayashi-Nomizu} : $R^{\nabla}(X,Y)Z + R^{\nabla}(Y,Z)X + R^{\nabla}(Z,X)Y =0$.

\begin{prop}
  \label{prop:Ricci symmetric}
  Let $\nabla$ be a torsion-free connection on $M$.
  The Ricci tensor field $Ric^{\nabla}$ is symmetric if and only if $\tr\{ Z \mapsto R^{\nabla}(X,Y)Z \}=0$ for $X,Y,Z \in \Gamma(TM)$.
\end{prop}

We will extend the concept of projective flatness and generalize Theorem~\ref{thm:condition of projectively flat} to our setting of coherent tangent bundles.

\vskip\baselineskip
\subsection{Projectively flat connection on coherent tangent bundle}
\label{sec:3-2}

Let $\E$ be a rank $n$ vector bundle on $M$, $\phi:TM \rightarrow \E$ a bundle map whose map over $M$.

\begin{dfn}
    \label{dfn:projectively_equivalence_on_coherent_tangent_bundle}
    Let $\nabla^{\E}$ and $\widehat{\nabla}^{\E}$ be connections on $\E$ that are relatively torsion-free with respect to $\phi$.
    Those two connections are called {\it projectively equivalent} with respect to $\phi$ if there exists $\rho\in \Gamma(\E^{*})$ such that
    \[
    \widehat{\nabla}^{\E}_{X}\eta = \nabla^{\E}_{X}\eta + \rho \circ \phi(X) \eta + \rho(\eta)\phi(X)
    \]
    holds for $X,Y\in\Gamma(TM)$ and $\eta\in\Gamma(\E)$.
    In particular, we say that $\nabla^{\E}$ is {\em projectively flat} with respect to $\phi$ if $\nabla^{\E}$ is projectively equivalent to some flat connection (i.e., a connection whose curvature tensor vanishes) with respect to $\phi$ in a neighborhood of an arbitrary point of $M$.
    Moreover, we say $(\E,\phi,\nabla^{\E})$ is projectively flat if $\nabla^{\E}$ is projectively flat with respect to $\phi$.
\end{dfn}

\begin{rem}
    If $\phi$ is an isomorphism, we define $\nabla$ and $\widehat{\nabla}$ by $\phi(\nabla_{X}Y) = \nabla^{\E}_{X} \phi(Y)$ and $\phi(\widehat{\nabla}_{X}Y) = \widehat{\nabla}^{\E}_{X} \phi(Y)$ for $X,Y\in\Gamma(TM)$.
    Then $\nabla^{\E}$ and $\widehat{\nabla}^{\E}$ are projectively equivalent with respect to $\phi$ if and only if $\nabla$ and $\widehat{\nabla}$ are projectively equivalent.
    This can be shown easily.
\end{rem}

We remark that projectively equivalent connections preserve pre-geodesics as mentioned in Section~\ref{sec:3-1}.
We verify that projective equivalence in Definition~\ref{dfn:projectively_equivalence_on_coherent_tangent_bundle} preserves a certain class of pre-geodesics.

\begin{dfn}
  [\cite{Nakajima-Ohmoto}]
  \label{dfn:pre-geodesic}
  Let $(\E,\phi,\nabla^{\E})$ be a coherent tangent bundle.
  A curve $c$ on $M$ is called a {\em pre-geodesic} of $(\E,\phi,\nabla^{\E})$ if it is an immersion $(\dot{c} \neq 0)$ and satisfies, for $t\in I$,
  \begin{align*}
    \phi\circ \dot{c}(t),\; \nabla^{\E}_{\dot{c}}(\phi\circ\dot{c})(t),\; \nabla^{\E}_{\dot{c}} \nabla^{\E}_{\dot{c}}(\phi\circ\dot{c})(t),\; \ldots
  \end{align*}
  are not simultaneously zero and any two are linearly dependent.
\end{dfn}

From this definition, we easily see that:

\begin{prop}
  Let $(\E,\phi,\nabla^{\E})$ and $(\E,\phi,\widehat{\nabla}^{\E})$ be two coherent tangent bundles.
  A curve $c$ is a pre-geodesic of $(\E,\phi,\nabla^{\E})$ if and only if $c$ is a pre-geodesic of $(\E,\phi,\widehat{\nabla}^{\E})$.
\end{prop}

A result analogous to Theorem~\ref{thm:condition of projectively flat} also holds for projective flatness on a coherent tangent bundle.
Theorems~\ref{thm:condition of projectively flat on coherent tangent bundle ver2} and \ref{thm:condition of projectively flat on coherent tangent bundle ver3} generalize Theorem~\ref{thm:condition of projectively flat}.
As a preliminary step, we first present Proposition~\ref{thm:condition of projectively flat on coherent tangent bundle}.

\begin{prop}
  \label{thm:condition of projectively flat on coherent tangent bundle}
    Let $(\E,\phi,\nabla^{\E})$ be a coherent tangent bundle, and assume that $\tr\{ \eta \mapsto R^{\nabla^{\E}}(X,Y)\eta \}=0$ for $X,Y\in\Gamma(TM)$ and $\eta\in\Gamma(\E)$.
    The following two conditions are mutually equivalent:
    \begin{itemize}
        \item[(i)]
            The connection $\nabla^{\E}$ is projectively flat.
        \item[(ii)]
            For any point $p\in M$, there exists a neighborhood $U$ of $p$ and there exists a tensor field $A\in \Gamma(T^{*}U \otimes \E^{*}|_{U})$ such that
            \begin{align}
              \label{eq:mainthm1_condition1}
              R^{\nabla^{\E}}(X,Y)\eta = \frac{1}{n-1} \{ A(Y,\eta)\phi(X) - A(X,\eta) \phi(Y) \}
            \end{align}
            and
            \begin{align}
              \label{eq:mainthm1_condition2}
              XA(Y,\eta) - YA(X,\eta) - A([X,Y],\eta) = A(Y,\nabla^{\E}_{X}\eta) - A(X, \nabla^{\E}_{Y}\eta)
            \end{align}
            are satisfied for $X,Y\in\Gamma(TU)$ and $\eta\in\Gamma(\E|_{U})$.
    \end{itemize}
\end{prop}

\begin{rem}
  By Proposition~\ref{prop:Ricci symmetric}, the condition of $\tr\{ \eta \mapsto R^{\nabla^{\E}}(X,Y)\eta \}=0$ is analogous to the condition that the Ricci tensor field is symmetric.
  We note that the assumption $\tr\{ \eta \mapsto R^{\nabla^{\E}}(X,Y)\eta \}=0$ is not used in the first part of the following proof.
  Therefore, ((i)$\Rightarrow$(ii)) holds without this assumption.
  Indeed, the proof of Theorem~\ref{thm:condition of projectively flat} also does.
\end{rem}

\begin{rem}
  The condition \eqref{eq:mainthm1_condition2} becomes unnecessary depending on the rank of $\phi$.
  For details, see Theorem~\ref{thm:condition of projectively flat on coherent tangent bundle ver3}.
\end{rem}

\noindent {\sl Proof of Proposition~\ref{thm:condition of projectively flat on coherent tangent bundle}} :\;
Let $U \subset M$ be a neighborhood of $p\in M$, $\rho \in \Gamma({\E^{*}\big|_{U}})$ a tensor field, and we define a connection $\widehat{\nabla}^{\E}$ on $\E\big|_{U}$ by
\begin{align*}
  \widehat{\nabla}^{\E}_{X}\eta = \nabla^{\E}_{X}\eta + \rho \circ \phi(X) \eta + \rho(\eta)\phi(X)
\end{align*}
for $X\in\Gamma(TU)$ and $\eta\in\Gamma(\E|_{U})$.
From straightforward calculations, we have, for $X,Y\in\Gamma(TU)$ and $\eta\in\Gamma(\E|_{U})$,
\begin{align}
  \label{curvature_of_projectively_equivalent_connection}
  \begin{split}
  R^{\widehat{\nabla}^{\E}} (X,Y)\eta &= R^{\nabla^{\E}}(X,Y)\eta - \{ Y\rho(\eta) - \rho(\nabla^{\E}_{Y}\eta) - \rho\circ\phi(Y) \rho(\eta)\}\phi(X)\\
  &\quad\;  + \{ X\rho(\eta) - \rho(\nabla^{\E}_{X}\eta) - \rho\circ\phi(X) \rho(\eta)\}\phi(Y) + (d(\rho \circ \phi)(X,Y))\eta.
  \end{split}
\end{align}

\vskip\baselineskip
\noindent
((i)$\Rightarrow$(ii))
Since $\nabla^{\E}$ is projectively flat, for any point $p\in M$, there exist a neighborhood $U$ and a tensor field $\rho \in \Gamma({\E^{*}\big|_{U}})$ such that the above connection $\widehat{\nabla}^{\E}$ on $\E|_{U}$ is flat.
We set
\[
A(X,\eta) := (n-1)\{ X\rho(\eta) -\rho(\nabla^{\E}_{X}\eta) -\rho \circ \phi(X) \rho(\eta) \}.
\]
By \eqref{curvature_of_projectively_equivalent_connection}, it holds that
\begin{align*}
  R^{\nabla^{\E}}(X,Y)\eta = \frac{1}{n-1} \{ A(Y,\eta)\phi(X) - A(X,\eta)\phi(Y) \} - (d(\rho \circ \phi) (X,Y))\eta
\end{align*}
for $X,Y\in\Gamma(TU)$ and $\eta\in\Gamma(\E|_{U})$.
Since $\tr \{\eta\mapsto R^{\nabla^{\E}}(X,Y)\eta \} =0$, it follows that
\begin{align*}
  0 &= \tr \{\eta\mapsto R^{\nabla^{\E}}(X,Y)\eta \}\\
  &= \frac{1}{n-1} \left\{ A(Y,\phi(X)) - A(X,\phi(Y)) \right\} - n (d(\rho \circ \phi) (X,Y))\\
  &= \{ Y\rho\circ\phi(X) - \rho(\nabla^{\E}_{Y}\phi(X)) -\rho\circ\phi(Y) \rho\circ\phi(X) \}\\
  &\qquad\qquad - \{ X\rho\circ\phi(Y) - \rho(\nabla^{\E}_{X}\phi(Y)) -\rho\circ\phi(X) \rho\circ\phi(Y) \} - n (d(\rho \circ \phi) (X,Y))\\
  &= -\left\{ X\rho\circ\phi(Y) - Y\rho\circ\phi(X) - \rho\circ\phi([X,Y]) \right\} - n (d(\rho \circ \phi) (X,Y))\\
  &= -(n+1) d(\rho \circ \phi) (X,Y).
\end{align*}
Hence, $\rho\circ\phi$ is a closed $1$-form, and \eqref{eq:mainthm1_condition1} is satisfied.
Moreover, \eqref{eq:mainthm1_condition2} follows from the following computation:
\begin{align*}
  \frac{1}{n-1} \{ XA&(Y,\eta) - YA(X,\eta) - A([X,Y],\eta) - A(Y,\nabla^{\E}_{X}\eta) + A(X,\nabla^{\E}_{Y}\eta) \}\\
  =& \{ XY\rho(\eta) - X\rho(\nabla^{\E}_{Y}\eta) - (X\rho\circ\phi(Y))\rho(\eta) -\rho\circ\phi(Y)(X\rho(\eta)) \} \\
  & \quad\quad - \{ YX\rho(\eta) - Y\rho(\nabla^{\E}_{X}\eta) - (Y\rho\circ\phi(X))\rho(\eta) - \rho\circ\phi(X)(Y\rho(\eta)) \} \\
  & \quad\quad\quad\quad - \{ [X,Y]\rho(\eta) -\rho(\nabla^{\E}_{[X,Y]}\eta) -\rho \circ \phi([X,Y]) \rho(\eta) \} \\
  & \quad\quad\quad\quad\quad\quad - \{ Y\rho(\nabla^{\E}_{X}\eta) -\rho(\nabla^{\E}_{Y}\nabla^{\E}_{X}\eta) -\rho \circ \phi(Y) \rho(\nabla^{\E}_{X}\eta) \} \\
  & \quad\quad\quad\quad\quad\quad\quad\quad + \{ X\rho(\nabla^{\E}_{Y}\eta) -\rho(\nabla^{\E}_{X}\nabla^{\E}_{Y}\eta) -\rho \circ \phi(X) \rho(\nabla^{\E}_{Y}\eta) \}\\
  =& - (d\rho\circ\phi(X,Y))\rho(\eta) - \rho\circ\phi(Y)\{ X\rho(\eta) - \rho(\nabla^{\E}_{X}\eta) \}\\
  & \quad\quad + \rho\circ\phi(X) \{ Y\rho(\eta) - \rho(\nabla^{\E}_{Y}\eta) \} - \rho(R^{\nabla^{\E}}(X,Y)\eta) \\
  =& - \rho\circ\phi(Y) \left( \frac{1}{n-1} A(X,\eta) + \rho\circ\phi(X) \rho(\eta) \right)\\
  & \quad\quad + \rho\circ\phi(X) \left( \frac{1}{n-1} A(Y,\eta) + \rho\circ\phi(Y) \rho(\eta) \right) - \rho(R^{\nabla^{\E}}(X,Y)\eta) \\
  =& \rho \left( \frac{1}{n-1} A(Y,\eta) \phi(X) - \frac{1}{n-1} A(X,\eta) \phi(Y) - R^{\nabla^{\E}}(X,Y)\eta \right) \\
  =& 0.
\end{align*}

\vskip\baselineskip
\noindent
((ii)$\Rightarrow$(i))
By the assumption, there exists $A\in \Gamma(T^{*}U \otimes \E^{*}|_{U})$ satisfying \eqref{eq:mainthm1_condition1} and \eqref{eq:mainthm1_condition2}.
Since $\tr \{\eta\mapsto R^{\nabla^{\E}}(X,Y)\eta \} =0$, the following equality holds for $X,Y\in\Gamma(TU)$ and $\eta\in\Gamma(\E|_{U})$:
\begin{align*}
  0 &= \tr \{\eta\mapsto R^{\nabla^{\E}}(X,Y)\eta \}\\
  &= \tr \left\{ \eta \mapsto \frac{1}{n-1} \{ A(Y,\eta)\phi(X) - A(X,\eta) \phi(Y) \right\}\\
  &= \frac{1}{n-1} A(Y,\phi(X)) - \frac{1}{n-1} A(X,\phi(Y)).
\end{align*}
Hence, for $X,Y\in\Gamma(TU)$, we obtain
\begin{align}
  \label{eq:A-sym}
  A(X,\phi(Y)) = A(Y,\phi(X)).
\end{align}

Assume that there exists $\rho \in \Gamma(\E^{*}|_{U})$ satisfying
\begin{align}
  \label{eq:diff_eq}
  X\rho(\eta) = \frac{1}{n-1} A(X,\eta) + \rho(\nabla^{\E}_{X}\eta) + \rho\circ\phi(X)\rho(\eta)
\end{align}
for $X\in\Gamma(TU)$ and $\eta\in\Gamma(\E|_{U})$.
Then, $\rho\circ\phi$ is closed.
Indeed, for $X,Y\in\Gamma(TU)$, it follows from \eqref{dfn:relatively torsion-free} that
\begin{align*}
  d(\rho\circ\phi)(X,Y) &= X(\rho\circ\phi(Y)) - Y(\rho\circ\phi(X)) - \rho\circ\phi([X,Y])\\
  &= \left\{ \frac{1}{n-1} A(X,\phi(Y)) + \rho(\nabla^{\E}_{X}\phi(Y)) + \rho\circ\phi(X)\rho\circ\phi(Y) \right\}\\
  &\qquad\qquad - \left\{ \frac{1}{n-1} A(Y,\phi(X)) + \rho(\nabla^{\E}_{Y}\phi(X)) + \rho\circ\phi(Y)\rho\circ\phi(X) \right\} \\
  & \qquad\qquad\qquad\qquad -\rho\circ\phi([X,Y])\\
  &= \rho(\nabla^{\E}_{X}\phi(Y) - \nabla^{\E}_{Y}\phi(X)) - \rho\circ\phi([X,Y])\\
  &=0.
\end{align*}
Moreover, by \eqref{eq:mainthm1_condition1} and \eqref{curvature_of_projectively_equivalent_connection}, the curvature tensor of the connection $\widehat{\nabla}^{\E}$ vanishes.
Hence, it suffices to show that such a $1$-form $\rho$ exists locally.

The following proof is based on Belgun~\cite{Belgun}, and we follow his notation.

Let $\barr{\nabla}$ be the connection on $\E^{*}|_{U}$ defined by $\barr{\nabla}_{X}\alpha=X\alpha - \alpha(\nabla_{X}\,\cdot\,)$ for $X\in\Gamma(TU)$ and $\alpha \in \Gamma(\E^{*}|_{U})$.
Set a $2n$-dimensional manifold $N$ to be the total space of $\E^{*}|_{U}$.
For $\alpha \in N$, the tangent space $T_{\alpha}N$ decomposes into the vertical space $T^{V}_{\alpha}N$ and the horizontal space $T^{H}_{\alpha}N$ with respect to $\barr{\nabla}$.
For $\alpha\in \E^{*}_{p}$, we have $T^{V}_{\alpha}N \simeq \E^{*}_{p}$ and $T^{H}_{\alpha}N \simeq T_{p}M$.
We denote by $\barr{X}\in\Gamma(T^{H}N)$ the horizontal lift of $X\in\Gamma(TM)$.
For any section $\alpha\in\Gamma(\E^{*}|_{U})$, the differential map of $\alpha :U \to N$, $p \mapsto \alpha_{p}$ satisfies that $\alpha_{*}(X)$ decomposes into the vertical part $\barr{\nabla}_{X}\alpha$ and the horizontal part $\barr{X}$.

We consider a distribution $\D$ on $N$ defined by
\begin{align*}
  \D_{\alpha} = \left\{ (\A^{V},\A^{H}) \in T^{V}_{\alpha}N \oplus T^{H}_{\alpha}N \;|\; \A^{V} = \frac{1}{n-1} A(\A^{H},\,\cdot\,) + \alpha \circ \phi (\A^{H}) \alpha  \right\}
\end{align*}
for $\alpha \in N$.
The distribution $\D$ is an $n$-dimensional subbundle of $TN$, since there is a one-to-one correspondence between vectors of $\D_{\alpha}$ and vectors of $T_{p}U$, where $\alpha \in \E^{*}_{p}$.
To check that $\D$ is involutive, take $\A,\B\in\Gamma(\D)$ satisfying $\A^{H}=\barr{X}$ and $\B^{H}=\barr{Y}$, where $X,Y\in\Gamma(TU)$.
We show that
\begin{align*}
  [\A,\B] = [\barr{X},\barr{Y}] + [\A^{V},\barr{Y}] + [\barr{X},\B^{V}] + [\A^{V},\B^{V}]
\end{align*}
is a section of $\D$.
First, the horizontal part of $[\barr{X},\barr{Y}]$ is $\barr{[X,Y]}$ and the vertical part is given by the curvature of $\barr{\nabla}$:
\begin{align*}
  [\barr{X},\barr{Y}]^{V} = -R^{\barr{\nabla}}(X,Y)\alpha = \alpha(R^{\nabla^{\E}}(X,Y)\,\cdot\,).
\end{align*}
Next, we compute $[\A^{V},\barr{Y}]$.
The bracket $[A(X,\,\cdot\,),\barr{Y}]$ is vertical and coincides with $-\barr{\nabla}_{Y} A(X,\,\cdot\,)$.
To compute $[\alpha \circ \phi(X) \alpha, \barr{Y}]$, we introduce the following notation:
\begin{align*}
  \rd_{i} &= \frac{\rd}{\rd x_{i}}, \quad
  \rd^{i} = \frac{\rd}{\rd y_{i}}, \quad
  X=X_{i}\rd_{i}, \quad Y=Y_{i}\rd_{i}, \quad
  y_{i}(\eta_{j}) = \delta_{ij},\\
  \phi(\rd_{i}) &=\phi_{ij}\eta_{j},\quad
  \alpha = \alpha_{i} y_{i}, \quad
  \nabla^{\E}_{X}\eta_{i} = \omega_{i}^{j}(X)\eta_{j}, \quad
  \barr{\nabla}_{X}y_{i} = \barr{\omega}_{i}^{j}(X)y_{j},
\end{align*}
where $(x_{1},\dots,x_{n})$ are local coordinates of $U$, $(\eta_{1},\dots,\eta_{n})$ is a frame of $\E|_{U}$, and $(y_{1},\dots,y_{n})$ is a frame of $\E^{*}|_{U}$.
The horizontal lift $\barr{X}$ is given by
\begin{align*}
  \barr{X}_{\alpha} = X_{i}\rd_{i} - \alpha_{j}\barr{\omega}_{j}^{k}(X)\rd^{k} = X_{i}\rd_{i} + \alpha_{j}\omega_{k}^{j}(X)\rd^{k}
\end{align*}
Therefore, the following equalities hold:
\begin{align*}
  [\alpha \circ \phi (X) \alpha,\;\barr{Y}] &= [X_{i}\phi_{i}\alpha_{i}\alpha_{j}\rd^{j},\; Y_{k}\rd_{k} + \alpha_{l}\omega_{m}^{l}(Y)\rd^{m}]\\
  &= [X_{i}\phi_{i}\alpha_{i}\alpha_{j}\rd^{j},\; Y_{k}\rd_{k}] + [X_{i}\phi_{i}\alpha_{i}\alpha_{j}\rd^{j},\; \alpha_{l}\omega_{m}^{l}(Y)\rd^{m}]\\
  &=-Y_{k}(\rd_{k}X_{i}\phi_{i})\alpha_{i}\alpha_{j} \rd^{j} - X_{i}\phi_{i}\alpha_{l}\omega_{i}^{l}(Y)\alpha_{j}\rd^{j}\\
  &= -\alpha(\nabla_{Y}\phi(X))\alpha.
\end{align*}
Then, we have
\begin{align*}
  [\A^{V},\barr{Y}] = \frac{1}{n-1}[A(X,\,\cdot\,),\barr{Y}] + [\alpha \circ \phi (X) \alpha,\;\barr{Y}] = -\frac{1}{n-1} \barr{\nabla}_{Y} A(X,\,\cdot\,) -\alpha(\nabla_{Y}\phi(X))\alpha.
\end{align*}
By the same argument, we obtain
\begin{align*}
  [\barr{X},\; \B^{V}] = \frac{1}{n-1} \barr{\nabla}_{X} A(Y,\,\cdot\,) + \alpha(\nabla_{X}\phi(Y)) \alpha.
\end{align*}
Moreover, let $\A^{V} = \A_{i}\rd^{i}$ and $\B^{V}=\B_{i}\rd^{i}$, then it follows that
\begin{align*}
  [\A^{V},\; \B^{V}] &= [\A_{i}\rd^{i},\; \B_{j}\rd^{j}]\\
  &= \A_{i}(\rd^{i}\B_{j})\rd^{j} - \B_{j}(\rd^{j}\A_{i})\rd^{i}\\
  &= \A_{i}(\rd^{i}Y_{k}\phi_{kl}\alpha_{l}\alpha_{j})\rd^{j} - \B_{j}(\rd^{j}X_{k}\phi_{kl}\alpha_{l}\alpha_{i})\rd^{i}\\
  &= \A_{i}Y_{k}\phi_{ki}\alpha_{j}\rd^{j} + \A_{i}Y_{k}\phi_{kl}\alpha_{l}\rd^{i} - \B_{j}X_{k}\phi_{kj}\alpha_{i}\rd^{i} - \B_{j}X_{k}\phi_{kl}\alpha_{l}\rd^{j}\\
  &= \A^{V}(\phi(Y))\alpha + \alpha\circ\phi(Y) \A^{V} - \B^{V}(\phi(X))\alpha - \alpha\circ\phi(X)\B^{V}\\
  &=\left\{ \frac{1}{n-1} A(X,\phi(Y)) + \alpha\circ\phi(X) \alpha\circ\phi(Y) \right\}\alpha\\
  & \qquad\qquad + \alpha\circ\phi(Y) \left\{ \frac{1}{n-1} A(X,\,\cdot\,) + \alpha\circ\phi(X)\alpha \right\}\\
  &\qquad\qquad\qquad\qquad - \left\{ \frac{1}{n-1} A(Y,\phi(X)) + \alpha\circ\phi(Y) \alpha\circ\phi(X) \right\}\alpha\\
  & \qquad\qquad\qquad\qquad\qquad\qquad - \alpha\circ\phi(X) \left\{ \frac{1}{n-1} A(Y,\,\cdot\,) + \alpha\circ\phi(Y)\alpha \right\}\\
  &= \left\{ \frac{1}{n-1} A(X,\phi(Y)) - \frac{1}{n-1} A(Y,\phi(X)) \right\}\alpha\\
  &\qquad\qquad + \frac{1}{n-1} A(X,\,\cdot\,)\alpha\circ\phi(Y) - \frac{1}{n-1} A(Y,\,\cdot\,)\alpha\circ\phi(X)\\
  &= -\alpha(R^{\nabla^{\E}}(X,Y)\,\cdot\,).
\end{align*}
The last equality follows from \eqref{eq:A-sym}.
Therefore, the horizontal part of $[\A,\B]$ is just the lift of the bracket
\begin{align*}
  [\A,\B]^{H} = \barr{[X,Y]},
\end{align*}
and the vertical part of $[\A,\B]$ is written as follows:
\begin{align*}
  [\A,\B]^{V} &= - \frac{1}{n-1} \barr{\nabla}_{Y}A(X,\,\cdot\,) - \alpha(\nabla^{\E}_{Y}\phi(X))\alpha + \frac{1}{n-1}\barr{\nabla}_{X}A(Y,\,\cdot\,) + \alpha(\nabla^{\E}_{X}\phi(Y))\alpha\\
  &= \frac{1}{n-1} \left\{ XA(Y,\,\cdot\,) - YA(X,\,\cdot\,) + A(X,\nabla^{\E}_{Y}\,\cdot\,) - A(Y,\nabla^{\E}_{X}\,\cdot\,) \right\}\\
  &\qquad\qquad\qquad\qquad\qquad\qquad\qquad\qquad\qquad + \alpha(\nabla^{\E}_{X}\phi(Y) - \nabla^{\E}_{Y}\phi(X))\alpha  \\
  &=\frac{1}{n-1}A([X,Y],\,\cdot\,) + \alpha\circ\phi([X,Y])\alpha.
\end{align*}
The last equality follows from \eqref{dfn:relatively torsion-free} and \eqref{eq:mainthm1_condition2}.
Hence, $[\A,\B]\in\Gamma(\D)$, and thus $\D$ is involutive.

Therefore, there exists a section $\rho \in \Gamma(\E^{*}|_{U})$ such that $\barr{\nabla}_{X}\rho \oplus \barr{X}\in\Gamma(\D)$, i.e., it holds that
\begin{align*}
  X\rho(\eta) - \rho(\nabla^{\E}_{X} \eta) = A(X,\eta) + \rho\circ\phi(X)\rho(\eta)
\end{align*} 
for $X\in\Gamma(TU)$ and $\eta\in\Gamma(\E|_{U})$, and $\nabla^{\E}$ is projectively flat.
\qed

\vskip\baselineskip

Under the assumption that $\rank\phi \geq 2$ almost everywhere, the condition (ii) can be considered globally.

\begin{thm}
  \label{thm:condition of projectively flat on coherent tangent bundle ver2}
    Let $(\E,\phi,\nabla^{\E})$ be a coherent tangent bundle.
    We assume that $\rank \phi \geq 2$ almost everywhere and $\tr\{ \eta \mapsto R^{\nabla^{\E}}(X,Y)\eta \}=0$ for $X,Y\in\Gamma(TM)$ and $\eta\in\Gamma(\E)$.
    The following two conditions are mutually equivalent:
    \begin{itemize}
        \item[(i)]
            The connection $\nabla^{\E}$ is projectively flat;
        \item[(ii)]
            There exists a tensor field $A\in \Gamma(T^{*}M \otimes \E^{*})$ such that \eqref{eq:mainthm1_condition1} and \eqref{eq:mainthm1_condition2} hold.
    \end{itemize}
\end{thm}

\proof
((i)$\Rightarrow$(ii))
By Proposition~\ref{thm:condition of projectively flat on coherent tangent bundle}, for each $i=1,2$ and $p_{i}\in M$, there exist a neighborhood $U_{i}$ of $p_{i}$ and a tensor field $A_{i}\in \Gamma(T^{*}U_{i} \otimes \E^{*}\big|_{U_{i}})$ satisfying \eqref{eq:mainthm1_condition1} and \eqref{eq:mainthm1_condition2}.
If $U_{1} \cap U_{2} \neq \emptyset$, then it holds that
\begin{align*}
  \{A_{1}(Y,\eta) - A_{2}(Y,\eta)\} \phi(X) = \{A_{1}(X,\eta) - A_{2}(X,\eta)\} \phi(Y)
\end{align*}
on $U_{1} \cap U_{2}$.
Since $\rank \phi \geq 2$ almost everywhere, it follows that $A_{1}=A_{2}$ on $U_{1}\cap U_{2}$.
Hence, we obtain a globally defined tensor field satisfying \eqref{eq:mainthm1_condition1} and \eqref{eq:mainthm1_condition2}.

\noindent
((ii)$\Rightarrow$(i))
This is clear from Proposition~\ref{thm:condition of projectively flat on coherent tangent bundle}.
\qed
\vskip\baselineskip

Furthermore, if we assume that $\rank \phi \geq 3$ almost everywhere, we obtain a stronger result.

\begin{thm}
  \label{thm:condition of projectively flat on coherent tangent bundle ver3}
  Let $(\E,\phi,\nabla^{\E})$ be a coherent tangent bundle with $\rank \phi \ge 3$ almost everywhere.
  Suppose that $\tr\{ \eta \mapsto R^{\nabla^{\E}}(X,Y)\eta \}=0$ for $X,Y\in\Gamma(TM)$ and $\eta\in\Gamma(\E)$.
  The following conditions are mutually equivalent.
  \begin{itemize}
      \item[(i)]
        The coherent tangent bundle $(\E,\phi,\nabla^{\E})$ is projectively flat;
      \item[(ii)]
        There exists a tensor field $A\in\Gamma(T^{*}M \otimes \E^{*})$ such that \eqref{eq:mainthm1_condition1} holds.
  \end{itemize}
\end{thm}

Theorem~\ref{thm:condition of projectively flat on coherent tangent bundle ver3} follows from the following lemma, which in turn follows from the second Bianchi identity.
The proof of the lemma is given in the appendix.

\begin{lem}
  \label{lem:Bianchi identity of coherent tangent bundle}
    Let $(\E,\phi,\nabla^{\E})$ be a coherent tangent bundle.
    We assume that $\rank \phi \ge 3$ almost everywhere, and there exists a tensor field $A\in\Gamma(T^{*}M\otimes\E^{*})$ such that \eqref{eq:mainthm1_condition1} holds.
    Then, \eqref{eq:mainthm1_condition2} follows from the second Bianchi identity.
\end{lem}

\vskip\baselineskip
\vskip\baselineskip
\vskip\baselineskip
\section{Relations of affine immersion and quasi-Codazzi structure}
\label{sec:4}

\subsection{Affine immersion}
\label{sec:4-1}
In this section, we summarize the submanifold theory of affine differential geometry.
For more details, see \cite{Dajczer-Tojeiro, Matsuzoe, Nomizu-Sasaki}.

\begin{dfn}
  [\cite{Nomizu-Sasaki}]
  Let $f:M \rightarrow \R^{n+1}$ be an immersion and $\xi$ a transversal vector field along $f$ such that for each $p\in M$,
  \[
  T_{f(p)}\R^{n+1} = f_{*}(T_{p}M) \oplus \operatorname{span}\{ \xi_{p} \},
  \]
  where $\operatorname{span}\{ \xi_{p} \}$ is the $1$-dimensional subspace spanned by $\xi_{p}$.
  We call the pair $\{ f, \xi \}$ an {\em affine immersion}.
\end{dfn}

Let $D$ be the standard connection on $\R^{n+1}$.
Given an affine immersion $\{ f,\xi \}$, from the decomposition of the tangent space, the covariant derivatives are decomposed as follows:
\begin{align}
  D_{X}f_{*}(Y) &= f_{*}(\nabla _{X}Y) + h(X,Y)\xi, \label{eq:the Gauss formula}\\
  D_{X}\xi &= -f_{*}(S(X)) + \tau(X)\xi, \label{eq:the Weingarten formula}
\end{align}
where $X,Y\in\Gamma(TM)$, $\nabla$ is a torsion-free connection on $M$, $h$ is a symmetric $(0,2)$-tensor field, $S$ a $(1,1)$-tensor field and $\tau$ is a $1$-form.
The formula \eqref{eq:the Gauss formula} is called the {\em Gauss formula}, and the formula \eqref{eq:the Weingarten formula} is called the {\em Weingarten formula}.
We call $\nabla$ the {\em induced connection}, $h$ the {\em second fundamental form}, $S$ the {\em affine shape operator}, and $\tau$ the {\em transversal connection form} of $\{f,\xi\}$.
Moreover, these quantities satisfy the following fundamental equations.
\vskip\baselineskip
\noindent
The {\em Gauss equation}\;:
\begin{align}
  R^{\nabla}(X,Y)Z = h(Y,Z)S(X) - h(X,Z)S(Y).
  \label{eq:the Gauss equation}
\end{align}
The {\em Codazzi equation} (I)\;:
\begin{align}
  (\nabla_{X}h)(Y,Z) + \tau(X)h(Y,Z) = (\nabla_{Y}h)(X,Z) + \tau(Y)h(X,Z).
  \label{eq:the Codazzi equation (I)}
\end{align}
The {\em Codazzi equation} (II)\;:
\begin{align}
  (\nabla_{X}S)(Y) - \tau(X)S(Y)= (\nabla_{Y}S)(X) - \tau(Y)S(X).
  \label{eq:the Codazzi equation (II)}
\end{align}
The {\em Ricci equation}\;:
\begin{align}
  h(X,S(Y)) - h(S(X),Y) = (\nabla_{X}\tau)(Y) - (\nabla_{Y}\tau)(X) = d\tau(X,Y).
  \label{eq:the Ricci equation}
\end{align}
Given an affine immersion, the above fundamental equations hold.
Conversely, the following proposition holds.

\begin{prop}
  [\cite{DNV}]
  \label{prop:exist the affine immersion}
  Let $M$ be a simply connected manifold, $h$ a symmetric $(0,2)$-tensor field, $\nabla$ a torsion-free connection, $S$ a $(1,1)$-tensor field, and $\tau$ a 1-form on $M$.
  If those quantities satisfy \eqref{eq:the Gauss equation}, \eqref{eq:the Codazzi equation (I)}, \eqref{eq:the Codazzi equation (II)}, and \eqref{eq:the Ricci equation}, then there exists an affine immersion that induces them.
\end{prop}

\begin{exam}
  [\cite{Nomizu-Sasaki}]
  \label{exam:graph immersion}
  For a function $\phi \in C^{\infty}(\R^{n})$, set an immersion $f:\R^{n}\rightarrow \R^{n+1}$ by
  \begin{align*}
    \R^{n} \ni \bx = (x_{1},\ldots, x_{n}) \mapsto f(\bx) = (x_{1},\ldots,x_{n},\phi(\bx))\in \R^{n+1}
  \end{align*}
  and a transverse vector field $\xi = (0,\ldots,0,1) \in \R^{n+1}$.
  We can easily check $\{ f,\xi \}$ is an affine immersion and call $\{ f,\xi \}$ a {\em graph immersion} with respect to $\phi$.
  This affine immersion induces
  \begin{align*}
    h \left( \frac{\rd}{\rd x_{i}}, \frac{\rd}{\rd x_{j}} \right) = \frac{\rd^{2}\phi}{\rd x_{i} \rd x_{j}}, \quad \nabla_{\frac{\rd}{\rd x_{i}}} \frac{\rd}{\rd x_{j}} = 0, \quad S\equiv 0 \quad \tau \equiv 0.
  \end{align*}
\end{exam}

\begin{exam}
  [\cite{Nomizu-Sasaki}]
  \label{exam:centro affine immersion}
  Let $f:M\rightarrow \R^{n+1}$ be an immersion and $o$ the origin of $\R^{n+1}$.
  If a position vector $f(\bx) = \overrightarrow{of(\bx)}\in \R^{n+1}$ is transverse to $f(M)$ for $\bx \in M$, we call $\{f,-f\}$ a {\em centroaffine immersion}.
  This affine immersion induces the following relations for $X,Y\in\Gamma(TM)$\;:
  \begin{align*}
    h (X,Y) = \frac{1}{n-1} Ric^{\nabla}(X,Y), \quad S= id, \quad \tau \equiv 0,
  \end{align*}
  where $\nabla$ is an induced connection with respect to $\{f,\xi\}$.
\end{exam}

If $h$ is nondegenerate, $f$ is said to be {\em nondegenerate}.
It is independent from the choice of transverse vector fields.
We define a volume form $\theta$ on $M$ by
\begin{align*}
  \theta (X_{1},\dots,X_{n}) = \omega(X_{1},\dots,X_{n},\xi),
\end{align*}
where $\omega$ is the standard volume form on $\R^{n+1}$ and $X_{1},\dots,X_{n}\in\Gamma(TM)$.
Then $\nabla_{X}\theta = \tau(X)\theta$ holds.
If $\nabla \theta \equiv 0$ (i.e., $\tau \equiv 0$), then $\{ f,\xi \}$ is called an {\em equiaffine immersion}.
A nondegenerate equiaffine immersion $\{f,\xi\}$ induces a Codazzi structure since the Codazzi equation (I) holds.
The dual connection is obtained as follows.

Let $(\R^{n+1})^{*}$ be the dual space of $\R^{n+1}$.
We define the pairing $\langle \; , \; \rangle$ of $(\R^{n+1})^{*}$ and  $\R^{n+1}$ by  $\langle \alpha, A \rangle:= \alpha (A)$ for $\alpha \in (\R^{n+1})^{*}$ and $A\in \R^{n+1}$.
We define a {\em conormal map} $\nu:M \ni p \mapsto \nu_{p} \in  (\R^{n+1})^{*}$ with respect to an affine immersion $\{ f,\xi \}$ by
\[
\langle \nu(p), f_{*}X_{p} \rangle = 0, \qquad \langle \nu(p), \xi_{p} \rangle = 1
\]
for $X_{p} \in T_{p}M$.
Differentiating the above equations, we get
\begin{align}
 \langle \nu_{*}X, f_{*}Y \rangle = -h(X,Y), \qquad \langle \nu_{*}X,\xi \rangle = -\tau(X) \label{eq:property of the conormal map}
\end{align}
for $X,Y\in\Gamma(TM)$.
Hence, $\nu:M \rightarrow (\R^{n+1})^{*}$ is an immersion if $h$ is nondegenerate, and $\nu$ is transversal to $\nu(M)$ itself, i.e., $\{ \nu, -\nu \}$ is a centroaffine immersion.
Then, by the Gauss formula \eqref{eq:the Gauss formula}, we obtain an induced torsion-free connection $\nabla^{*}$ which satisfies
\[
Xh(Y,Z) = h(\nabla_{X}Y,Z) + h(Y,\nabla^{*}_{X}Z) + \tau(Z)h(X,Y)
\]
for $X,Y,Z\in\Gamma(TM)$.
Therefore, $(h,\nabla,\nabla^{*})$ is a Codazzi structure if $\tau \equiv 0$.
Hence, a nondegenerate equiaffine immersion $\{f,\xi\}$ induces a Codazzi structure.
The converse is known to hold by the following theorem.

\begin{thm}
  [\cite{DNV}]
  \label{thm:nondegenerate and equiaffine <=> Codazzi structure}
  Let $\{ f,\xi \}$ be a nondegenerate equiaffine immersion, $h$ the second fundamental form, and $\nabla$ the induced connection of $\{f,\xi\}$.
  In this case, $(h,\nabla,\nabla^{*})$ is a Codazzi structure and the dual connection $\nabla^{*}$ is a projectively flat connection with symmetric Ricci tensor field.
  Conversely, if $(h,\nabla,\nabla^{*})$ is a Codazzi structure on a simply connected manifold and the dual connection $\nabla^{*}$ is a projectively flat connection with symmetric Ricci tensor field, then there exists a nondegenerate equiaffine immersion that induces the Codazzi structure.
\end{thm}

\begin{exam}
  Let $\{f:\R^{2}\rightarrow \R^{3},\xi\}$ be a graph immersion with respect to $\phi(x_{1},x_{2})= \frac{1}{2}(x_{1}^{2}+x_{2}^{2})$.
  The second fundamental form $h$ is
  \[
  h \left( \frac{\rd}{\rd x_{i}}, \frac{\rd}{\rd x_{j}} \right) = \delta_{ij},
  \]
  where $\delta_{ij}$ is the Kronecker delta for $i,j=1,2$.
  Therefore, since $h$ is nondegenerate, the graph immersion $\{f,\xi\}$ with respect to $\phi$ induces a Hessian structure.
\end{exam}

\begin{exam}
  We set an open disk $D_{1}=\{ (x_{1},x_{2})\in\R^{2} \big| x_{1}^{2}+x_{2}^{2} <1 \}$ and $f:D_{1}\ni (x_{1},x_{2}) \mapsto (x_{1},x_{2},x_{1}^{2} + x_{2}^{2} +1 )\in\R^{3}$.
  Then $\{f,-f\}$ is a centroaffine immersion and we have
  \begin{align*}
    h\left( \frac{\rd}{\rd x_{i}}, \frac{\rd}{\rd x_{j}} \right) = \frac{2 \delta_{ij}}{1 - x_{1}^{2} - x_{2}^{2}}, \qquad \nabla_{\frac{\rd}{\rd x_{i}}}\frac{\rd}{\rd x_{j}} = \frac{2\delta_{ij}}{1- x_{1}^{2} - x_{2}^{2}} \left( x_{1} \frac{\rd}{\rd x_{1}} + x_{2} \frac{\rd}{\rd x_{2}} \right)
  \end{align*}
  for $i,j=1,2$.
  Since $h$ is nondegenerate, the centroaffine immersion induces a Codazzi structure.
\end{exam}

\begin{dfn}
  [\cite{Kurose}]
  Let $\{ f,\xi \}$ be an affine immersion and $\nu$ a conormal map of $\{ f,\xi \}$.
  We define a function $\rho^{G}$ on $M\times M$ by
  \[
  \rho^{G}(p,q) := \langle \nu(q),f(p)-f(q)\rangle
  \]
  for every $p,q\in M$.
  We call $\rho^{G}$ the {\em geometric divergence} of $\{f,\xi\}$.
\end{dfn}

In the original paper \cite{Kurose}, the geometric divergence is defined for a nondegenerate equiaffine immersion.
However, it can also be defined for any affine immersion.

By straightforward calculations, we can check the geometric divergence of a nondegenerate equiaffine immersion is a contrast function of the induced Codazzi structure.

Proposition~\ref{prop:exist the affine immersion} is used in the proof of Theorem~\ref{thm:nondegenerate and equiaffine <=> Codazzi structure}.
We note that Proposition~\ref{prop:exist the affine immersion} still holds when a second fundamental form $h$ is degenerate.
Moreover, a conormal map $\nu$ and a geometric divergence $\rho^{G}$ exist even if $h$ is degenerate.
In Section~\ref{sec:4-2}, we generalize Theorem~\ref{thm:nondegenerate and equiaffine <=> Codazzi structure} under the weaker assumption that $h$ is not necessarily nondegenerate, and we also investigate the relationship between quasi-Codazzi structures and geometric divergences induced by equiaffine immersions.

\vskip\baselineskip
\subsection{Quasi-Codazzi structure on equiaffine immersion}
\label{sec:4-2}

Let $\{ f:M\rightarrow \R^{n+1}, \xi \}$ be an equiaffine immersion and $\nu$ the conormal map of $\{ f,\xi \}$.
Since $\nu$ is not necessarily an immersion, a dual connection cannot be defined.
Nevertheless, we show that there exists a dual connection in the sense of Definition~\ref{dfn:dual connection}, and a quasi-Codazzi structure is induced by $\{ f,\xi \}$ in this situation.

Let a torsion-free connection $\nabla$ and a symmetric $(0,2)$-tensor field $h$ be defined by  the Gauss formula \eqref{eq:the Gauss formula} of $\{ f,\xi \}$.
Set a vector space $E_{p}$ for $p\in M$,
\[
E_{p} := E^{+}_{p} \oplus E^{-}_{p} := f_{*}(T_{p}M) \oplus (f_{*}(T_{p}M))^{*} \subset T_{f(p)}\R^{n+1} \oplus T_{f(p)}(\R^{n+1})^{*}
\]
and a rank $2n$ vector bundle $E:= \bigcup_{p\in M} E_{p}$ on $M$.
We define a metric $\theta$ and a $(1,1)$-tensor field $I$ on $E$ by
\begin{align*}
  \theta_{p}(\eta^{+}_{p} \oplus \eta^{-}_{p}, \zeta^{+}_{p}\oplus \zeta^{-}_{p}) &:= \frac{1}{2} \{ \langle \zeta^{-}_{p}, \eta^{+}_{p} \rangle + \langle \eta^{-}_{p}, \zeta^{+}_{p} \rangle \},\\
  I_{p} (\eta^{+}_{p}\oplus \eta^{-}_{p}) &:= \eta^{+}_{p} \oplus - \eta^{-}_{p}
\end{align*}
for $p\in M$ and $\eta^{+}_{p}\oplus \eta^{-}_{p}, \zeta^{+}_{p}\oplus\zeta^{-}_{p} \in E_{p}$, where $\langle \eta^{-}_{p} , \eta^{+}_{p} \rangle = \eta^{-}_{p}(\eta^{+}_{p})$.
We can easily check $(E,\theta,I)$ is a para-Hermitian vector bundle on $M$.
We set a connection $\nabla^{+}$ on $E^{+}$ by
\[
\nabla^{+}_{X}f_{*}Y = f_{*}(\nabla_{X}Y),
\]
a connection $\nabla^{-}$ on $E^{-}$ by
\[
X\theta(\eta^{+},\zeta^{-}) = \theta(\nabla^{+}_{X}\eta^{+},\zeta^{-}) + \theta(\eta^{+},\nabla^{-}_{X}\zeta^{-}),
\]
and a bundle map $\Phi:TM\rightarrow E=E^{+}\oplus E^{-}$ by
\[
\Phi(X) = f_{*}X \oplus -\nu_{*}X \in \Gamma (E^{+}\oplus E^{-})
\]
for $X,Y\in\Gamma(TM)$, $\eta^{+}\in \Gamma(E^{+})$ and $\zeta^{-}\in \Gamma(E^{-})$.
By the condition \eqref{eq:property of the conormal map}, we get
\begin{align*}
  \theta (\Phi(X),\Phi(Y)) &= \theta(f_{*}X \oplus -\nu_{*}X, f_{*}Y\oplus -\nu_{*}Y)\\
  & = \frac{1}{2} \{ \langle-\nu_{*}Y,f_{*}X \rangle + \langle-\nu_{*}X,f_{*}Y \rangle \}\\
  &=\frac{1}{2} \{ h(Y,X) + h(X,Y) \}\\
  &= h(X,Y).
\end{align*}
Let us prove that $(h,(E,\theta,I),\Phi,\nabla^{+},\nabla^{-})$ is a quasi-Codazzi structure on $M$, namely, $\nabla^{+}$ and $\nabla^{-}$ are relatively torsion-free with respect to $f_{*}$ and $- \nu_{*}$.
Clearly, $\nabla^{+}$ is relatively torsion-free with respect to $f_{*}$ since $\nabla$ is torsion-free.
We consider the $(0,3)$-tensor field $C$ defined by \eqref{dfn:cubic tensor of quasi-Codazzi}:
\[
C(X,Y,Z) = -2 \{ \theta(\nabla^{+}_{X}f_{*}Y, -\nu_{*}Z) - \theta(f_{*}Z, \nabla^{-}_{X} (-\nu_{*}Y)) \}.
\]
The $(0,3)$-tensor field $C$ coincides with $\nabla h$.
Indeed, it holds that
\begin{align*}
  \nabla_{X}h(Y,Z) &= Xh(Y,Z) - h(\nabla_{X}Y,Z) - h(Y,\nabla_{X}Z)\\
  &= 2\{ X\theta(f_{*}Y,-\nu_{*}Z) - \theta(f_{*}(\nabla_{X}Y),-\nu_{*}Z) -\theta(f_{*}(\nabla_{X}Z),-\nu_{*}Y) \}\\
  &= 2\{ X\theta(f_{*}Y,-\nu_{*}Z) -\theta(\nabla^{+}_{X}f_{*}Y,-\nu_{*}Z) -\theta(\nabla^{+}_{X}f_{*}Z,-\nu_{*}Y) \}\\
  &= 2\{ \theta(f_{*}Z,\nabla^{-}_{X}(-\nu_{*}Y))-\theta(\nabla^{+}_{X}f_{*}Y,-\nu_{*}Z) \}\\
  &= C(X,Y,Z).
\end{align*}
Then $C$ is totally symmetric by the Codazzi equation \eqref{eq:the Codazzi equation (I)} and the assumption that $\{f,\xi\}$ is equiaffine.
By Proposition~\ref{prop of quasi-Codazzi str}, $\nabla^{-}$ is relatively torsion-free with respect to $-\nu_{*}$.
Hence, $(h,(E,\theta,I),\Phi,\nabla^{+},\nabla^{-})$ is a quasi-Codazzi structure on $M$.
Therefore, an equiaffine immersion induces the quasi-Codazzi structure $(h,(E,\theta,I),\Phi,\nabla^{+},\nabla^{-})$, which we call the {\it induced quasi-Codazzi structure} of $\{f,\xi\}$.

\begin{thm}\upshape
  \label{main_thm_1}
  The quasi-Codazzi structure $(h,(E,\theta,I),\Phi,\nabla^{+},\nabla^{-})$ induced by an equiaffine immersion $\{f,\xi\}$ satisfies that $\nabla^{-}$ is projectively flat and $\tr\{ \eta^{-} \mapsto R^{\nabla^{-}}(X,Y)\eta^{-} \} = 0$ for $X,Y\in\Gamma(TM)$ and $\eta^{-}\in\Gamma(E^{-})$.
\end{thm}

Conversely, we can show that if a quasi-Codazzi structure satisfies certain additional conditions, there exists an equiaffine immersion such that it induces the quasi-Codazzi structure.

\begin{thm}\upshape
  \label{main_thm_2}
  Let $(h,(E,\theta,I),\Phi,\nabla^{+},\nabla^{-})$ be a quasi-Codazzi structure on a simply connected manifold.
  Suppose that $\Phi^{+}$ is a bundle isomorphism and $\rank \Phi^{-} \geq 2$ almost everywhere.
  If $\nabla^{-}$ is projectively flat and $\tr\{\eta^{-} \mapsto R^{\nabla^{-}}(X,Y)\eta^{-}\} = 0$ for $X,Y\in\Gamma(TM)$ and $\eta^{-}\in\Gamma(E^{-})$, then there exists an equiaffine immersion such that the induced quasi-Codazzi structure is isomorphic to the given one.
\end{thm}

In order to prove Theorems~\ref{main_thm_1} and \ref{main_thm_2}, we begin with some lemmas.
Let $(h,(E,\theta,I),\Phi,\nabla^{+},\nabla^{-})$ be a quasi-Codazzi structure with an isomorphism $\Phi^{+}$ and $\nabla$ a connection on $M$ defined by $\Phi^{+}(\nabla_{X}Y)=\nabla^{+}_{X}\Phi^{+}(Y)$ for $X,Y\in\Gamma(TM)$.
We take a tensor field $A\in\Gamma(T^{*}M\otimes E^{-})$ and a $(1,1)$-tensor field $S\in\Gamma(TM\otimes T^{*}M)$ satisfying
\begin{align}
  \label{dfn:S}
  2\theta(\Phi^{+}(S(X)),\eta^{-}) = \frac{1}{n-1}A(X,\eta^{-})
\end{align}
for $X\in\Gamma(TM)$ and $\eta\in\Gamma(E^{-})$.

\begin{lem}
  \label{lem:1 to prove thm}
  The following conditions are equivalent:
  \begin{itemize}
      \item[(i)]
          For $X,Y\in\Gamma(TM)$ and $\eta^{-}\in\Gamma(E^{-})$,
          \begin{align}
            \label{lem1:condition1}
            R^{\nabla^{-}}(X,Y)\eta^{-} = \frac{1}{n-1} \{ A(Y,\eta^{-})\Phi^{-}(X) - A(X,\eta^{-})\Phi^{-}(Y) \}.
          \end{align}
      \item[(ii)]
          For $X,Y,Z\in\Gamma(TM)$,
          \begin{align*}
            R^{\nabla}(X,Y)Z = h(Y,Z)S(X) - h(X,Z)S(Y).
          \end{align*}
  \end{itemize}
\end{lem}

\proof
Suppose that the condition $\text{(i)}$ holds.
We have
\begin{align*}
  \theta(\Phi^{+}(&R^{\nabla}(X,Y)Z), \eta^{-})\\
  &= \theta(R^{\nabla^{+}}(X,Y)\Phi^{+}(Z),\eta^{-})\\
  &= -\theta(\Phi^{+}(Z), \; R^{\nabla^{-}}(X,Y)\eta^{-})\\
  &= -\theta(\Phi^{+}(Z),\; 2\theta( \Phi^{+}(S(Y)),\eta^{-} )\Phi^{-}(X) - 2\theta(\Phi^{+}(S(X)), \eta^{-} )\Phi^{-}(Y)  )\\
  &= -2\theta(\Phi^{+}(Z),\Phi^{-}(X)) \theta(\Phi^{+}(S(Y)),\eta^{-})\\
  &\quad\quad\quad\quad\quad\quad\quad\quad + 2\theta(\Phi^{+}(Z),\Phi^{-}(Y)) \theta(\Phi^{+}(S(X)),\eta^{-})\\
  &= -h(Z,X)\theta(\Phi^{+}(S(Y)),\eta^{-}) + h(Z,Y)\theta(\Phi^{+}(S(X)),\eta^{-})\\
  &= \theta(\;\Phi^{+}(h(Y,Z)S(X) - h(X,Z)S(Y)),\;\eta^{-} ).
\end{align*}
The second equality follows from the properties of the dual connection (see Proposition~3.16 in \cite{Kayo}).
Therefore, the condition $\text{(i)}$ implies $\text{(ii)}$, and vice versa.
\qed

\begin{lem}
  \label{lem:2 to prove thm}
  The following conditions are equivalent:
  \begin{itemize}
      \item[(i)]
          For $X,Y\in\Gamma(TM)$ and $\eta^{-}\in\Gamma(E^{-})$,
          \begin{align}
            \label{lem2:condition1}
            XA(Y,\eta^{-}) - YA(X,\eta^{-}) - A([X,Y],\eta^{-}) = A(Y,\nabla^{-}_{X}\eta^{-}) - A(X,\nabla^{-}_{Y}\eta^{-}).
          \end{align}
      \item[(ii)]
          For $X,Y,Z\in\Gamma(TM)$,
          \begin{align*}
            (\nabla_{X}S)(Y) = (\nabla_{Y}S)(X).
          \end{align*}
  \end{itemize}
\end{lem}

\proof
It follows that
\begin{align*}
  2\theta(\Phi^{+}&((\nabla_{X}S)(Y) - (\nabla_{Y}S)(X)) , \eta^{-})\\
   &= 2\theta(\Phi^{+}(\nabla_{X}S(Y) - S(\nabla_{X}Y) - \nabla_{Y}S(X) + S(\nabla_{Y}X)), \eta^{-})\\
  &= 2\theta(\nabla^{+}_{X}\Phi^{+}(S(Y)), \eta^{-}) - 2\theta(\nabla^{+}_{Y}\Phi^{+}(S(X)), \eta^{-}) - 2\theta(\Phi^{+}(S(\nabla_{X}Y -\nabla_{Y}X)), \eta^{-})\\
  &= 2\{ X\theta(\Phi^{+}(S(Y)),\eta^{-}) - \theta(\Phi^{+}(S(Y)), \nabla^{-}_{X}\eta^{-}) \}\\
  &\qquad -2\{ Y\theta(\Phi^{+}(S(X)),\eta^{-}) - \theta(\Phi^{+}(S(X)), \nabla^{-}_{Y}\eta^{-}) \} - 2\theta(\Phi^{+}(S([X,Y])), \eta^{-})\\
  &=\frac{1}{n-1} \{ XA(Y,\eta^{-}) - A(Y,\nabla^{-}_{X}\eta^{-}) -YA(X,\eta^{-}) + A(X,\nabla^{-}_{Y}\eta^{-}) - A([X,Y],\eta^{-})\}.
\end{align*}
Therefore, the condition $\text{(i)}$ implies $\text{(ii)}$, and vice versa.
\qed

\begin{lem}
  \label{lem:3 to prove thm}
  Suppose that
  \[
  R^{\nabla^{-}}(X,Y)\eta^{-} = \frac{1}{n-1} \{ A(Y,\eta^{-})\Phi^{-}(X) - A(X,\eta^{-})\Phi^{-}(Y) \}
  \]
  for $X,Y\in\Gamma(TM)$ and $\eta^{-}\in\Gamma(E^{-})$.
  The following conditions are equivalent:
  \begin{itemize}
        \item[(i)]
            For $X,Y\in\Gamma(TM)$ and $\eta^{-}\in\Gamma(E^{-})$,
            \begin{align}
              \label{lem3:condition1}
              tr\{\eta^{-} \mapsto R^{\nabla^{-}}(X,Y)\eta^{-}\} = 0.
            \end{align}
        \item[(ii)]
            For $X,Y\in\Gamma(TM)$,
            \begin{align*}
              h(S(X),Y) = h(S(Y),X).
            \end{align*}
    \end{itemize}
\end{lem}

\proof
Let $\{e_{1}^{-}, \dots , e_{n}^{-}\}$ be a frame of $E^{-}$, and define $\Phi_{i}^{-}$ for $i\in\{1,\dots,n\}$ by $\Phi^{-}(X) = \Phi_{i}^{-}(X) e^{-}_{i}$.
Then it holds that
\begin{align*}
  R^{\nabla^{-}}(X,Y)e_{i}^{-} &= \frac{1}{n-1} \{ A(Y,e_{i}^{-})\Phi^{-}(X) - A(X,e_{i}^{-})\Phi^{-}(Y) \}\\
  &= 2 \theta(\Phi^{+}(S(Y)),e_{i}^{-})\Phi^{-}(X) - 2 \theta(\Phi^{+}(S(X)),e_{i}^{-})\Phi^{-}(Y)\\
  &= 2 \theta(\Phi^{+}(S(Y)),e_{i}^{-})\Phi_{j}^{-}(X) e_{j}^{-} - 2 \theta(\Phi^{+}(S(X)),e_{i}^{-})\Phi_{j}^{-}(Y) e_{j}^{-}.
\end{align*}
Hence, we have
\begin{align*}
  \tr \{ \eta^{-} \mapsto R^{\nabla^{-}}(X,Y)\eta^{-} \} &= 2 \theta(\Phi^{+}(S(Y)),e_{i}^{-})\Phi_{i}^{-}(X) - 2 \theta(\Phi^{+}(S(X)),e_{i}^{-})\Phi_{i}^{-}(Y)\\
  &= 2\theta(\Phi^{+}(S(Y)), \Phi^{-}(X)) - 2\theta(\Phi^{+}(S(X)), \Phi^{-}(Y))\\
  &= h(S(Y),X) - h(S(X),Y).
\end{align*}
Therefore, the condition $\text{(i)}$ implies $\text{(ii)}$, and vice versa.
\qed

\vskip\baselineskip
\noindent {\sl Proof of Theorem~\ref{main_thm_1}} :
Let $\{f,\xi\}$ be an equiaffine immersion and $(h, (E,\tau,I),\Phi,\nabla^{+},\nabla^{-})$ the induced quasi-Codazzi structure of $\{f,\xi\}$.
The fundamental equations \eqref{eq:the Gauss equation}, \eqref{eq:the Codazzi equation (I)}, \eqref{eq:the Codazzi equation (II)}, and \eqref{eq:the Ricci equation} hold.
Then, $\nabla^{-}$ is projectively flat by Proposition~\ref{thm:condition of projectively flat on coherent tangent bundle} and Lemmas~\ref{lem:1 to prove thm}, \ref{lem:2 to prove thm}, and \ref{lem:3 to prove thm}.
\qed

\vskip\baselineskip
\noindent {\sl Proof of Theorem~\ref{main_thm_2}} :
Let $(h, (E,\tau,I),\Phi,\nabla^{+},\nabla^{-})$ be a quasi-Codazzi structure on a simply connected manifold.

By the assumption, there exists a tensor field $A\in\Gamma(T^{*}M\otimes (E^{-})^{*})$ satisfying \eqref{lem1:condition1}, \eqref{lem2:condition1} and \eqref{lem3:condition1} by Theorem~\ref{thm:condition of projectively flat on coherent tangent bundle ver2}.
We set a $(1,1)$-tensor field $S\in\Gamma(TM\otimes T^{*}M)$ by \eqref{dfn:S}, a connection $\nabla$ on $M$ by $\Phi^{+}(\nabla_{X}Y) = \nabla^{+}_{X}\Phi^{+}(Y)$, and a $(0,1)$-tensor field $\tau\in\Gamma(T^{*}M)$ by $\tau \equiv 0$.

Let us check that $h$, $\nabla$, $S$ and $\tau$ satisfy the fundamental equations.
By Lemmas~\ref{lem:1 to prove thm}, \ref{lem:2 to prove thm} and \ref{lem:3 to prove thm}, the Gauss equation \eqref{eq:the Gauss equation}, the Codazzi equation \text{(I)} \eqref{eq:the Codazzi equation (I)}, the Codazzi equation \text{(II)} \eqref{eq:the Codazzi equation (II)} and the Ricci equation \eqref{eq:the Ricci equation} are satisfied.
Moreover, we have the Codazzi equation \text{(I)} \eqref{eq:the Codazzi equation (I)}.
Indeed, $\nabla h$ is totally symmetric because the following identity holds for the totally symmetric $(0,3)$-tensor $C$ defined by \eqref{dfn:cubic tensor of quasi-Codazzi}:
\begin{align*}
  C(X,Y,Z) &= -2\{ \theta(\nabla^{+}_{X}\Phi^{+}(Y),\Phi^{-}(Z)) - \theta(\Phi^{+}(Z), \nabla^{-}_{X}\Phi^{-}(Y)) \}\\
  &= -2\{ \theta(\nabla^{+}_{X}\Phi^{+}(Y),\Phi^{-}(Z))\\
  &\quad\quad\quad\quad\quad- X\theta(\Phi^{+}(Z),\Phi^{-}(Y)) + \theta(\nabla^{+}_{X}\Phi^{+}(Z),\Phi^{-}(Y)) \}\\
  &= -2\{ \theta(\Phi^{+}(\nabla_{X}Y),\Phi^{-}(Z))\\
  &\quad\quad\quad\quad\quad-X\theta(\Phi^{+}(Z),\Phi^{-}(Y)) + \theta(\Phi^{+}(\nabla_{X}Z),\Phi^{-}(Y)) \}\\
  &= -h(\nabla_{X}Y,Z) + Xh(Z,Y) - h(\nabla_{X}Z,Y)\\
  &= (\nabla_{X} h)(Y,Z)
\end{align*}
for $X,Y,Z\in\Gamma(TM)$.
Therefore, the fundamental equations are satisfied.
By Proposition~\ref{prop:exist the affine immersion}, there exists an affine immersion $\{f,\xi\}$ inducing $h$, $\nabla$, $S$ and $\tau$.

Let $(h,(\til{E}= \til{E}^{+}\oplus \til{E}^{-} = f_{*}(TM) \oplus (f_{*}(TM))^{*} ,\til{\theta},\til{I}), \til{\Phi}=f_{*}\oplus -\nu_{*}, \til{\nabla}^{+}, \til{\nabla}^{-})$ be the induced quasi-Codazzi structure of $\{f,\xi \}$, where $\nu$ is the conormal map of $\{f,\xi\}$.
We now verify that this structure and $(h,(E,\theta,I),\Phi,\nabla^{+},\nabla^{-})$ are isomorphic.
We define a bundle map $\F^{+}:E^{+}\rightarrow \til{E}^{+}$ by
\[
\F^{+}(\eta^{+}) = f_{*} \circ (\Phi^{+})^{-1}(\eta^{+}),
\]
and a bundle map $\F^{-}:E^{-}\rightarrow \til{E}^{-}$ by the condition
\[
\til{\theta}(\F^{+}(\eta^{+}), \F^{-}(\eta^{-})) = \theta(\eta^{+},\eta^{-})
\]
for $\eta^{+}\in\Gamma(E^{+})$ and $\eta^{-}\in\Gamma(E^{-})$.
Then the bundle map $\F = \F^{+}\oplus \F^{-} : E\rightarrow \til{E}$ is a bundle isomorphism.
A direct computation shows that $\F\circ \Phi = \til{\Phi}$, $\F\circ I = \til{I}$ and $\F(\nabla^{+}_{X}\eta^{+}) = \til{\nabla}^{+}_{X} \F(\eta^{+})$ are satisfied.
In addition, we have
\begin{align*}
  \til{\theta}(\F^{+}(\eta^{+}), \F^{-}(\nabla^{-}_{X}\eta^{-}) ) &= \theta(\eta^{+},\nabla^{-}_{X}\eta^{-})\\
  &= X\theta(\eta^{+},\eta^{-}) - \theta(\nabla^{+}_{X}\eta^{+},\eta^{-})\\
  &= X\til{\theta} (\F^{+}(\eta^{+}), \F^{-}(\eta^{-})) - \til{\theta} (\F^{+}(\nabla^{+}_{X}\eta^{+}), \F^{-}(\eta^{+}))\\
  &= X\til{\theta}(\F^{+}(\eta^{+}),\F^{-}(\eta^{-})) - \til{\theta} (\til{\nabla}^{+}_{X}\F^{+}(\eta^{+}), \F(\eta^{-}))\\
  &= \til{\theta} (\F^{+}(\eta^{+}), \til{\nabla}^{-}_{X}\F(\eta^{-})).
\end{align*}
Therefore, $\F(\nabla^{-}_{X}\eta^{-}) = \til{\nabla}^{-}_{X} \F(\eta^{-})$ holds.
Hence, the two quasi-Codazzi structures are isomorphic by $\F$.
\qed

\begin{exam}
  Let $\{f,\xi\}$ be a graph immersion with respect to $\phi$.
  Since the second fundamental form $h$ coincides with the Hessian of $\phi$, $h$ is degenerate if $\phi$ is not a strictly convex function.
  Furthermore, the induced connection of $\{f,\xi\}$ is flat.
  Then a graph immersion induces a quasi-Hessian structure as described above.

  Let $n=2$ and define $\phi(x_{1},x_{2}) = \frac{1}{6}(x_{1}^{3} + x_{2}^{3})$.
  In this case, the second fundamental form $h$ is given by
  \[
  h\left( \frac{\rd}{\rd x_{i}}, \frac{\rd}{\rd x_{j}} \right) = \frac{\rd^{2}\phi}{\rd x_{i} \rd x_{j}} = \delta_{ij} x_{i}
  \]
  for $i,j=1,2$, where $\delta_{ij}$ is the Kronecker delta.
  Hence, $h$ is degenerate where $x_{1}=0$ or $x_{2}=0$.
  Therefore, the graph immersion $\{f,\xi\}$ with respect to $\phi$ induces a quasi-Hessian structure, but not a Hessian structure.
  We set $\widehat{\alpha}_{i} := \alpha_{i}|_{f_{*}T\R^{2}} \in \Gamma((f_{*}T\R^{2})^{*})$ for $i=1,2$, where $\alpha_{i} \in \Gamma(T^{*}\R^{3})$ satisfies $\alpha_{i}(\frac{\rd}{\rd x_{j}}) = \delta_{ij}$ for $i,j=1,2,3$.
  The conormal map of $\{f,\xi\}$ is obtained by $\nu:\R^{2}\ni (x_{1},x_{2})\mapsto (-\frac{x_{1}^{2}}{2},-\frac{x_{2}^{2}}{2},1)\in(\R^{3})^{*}$.
  The induced quasi-Hessian structure $(h,(E,\theta,I),\Phi,\nabla^{+},\nabla^{-})$ is given by
  \begin{gather*}
    E = f_{*}T\R^{2} \oplus (f_{*}T\R^{2})^{*},\\
    \theta \left( f_{*}\frac{\rd}{\rd x_{i}} \oplus \widehat{\alpha}_{j}, f_{*}\frac{\rd}{\rd x_{k}} \oplus \widehat{\alpha}_{l} \right) = \widehat{\alpha}_{l}\left( f_{*}\frac{\rd}{\rd x_{i}} \right) + \widehat{\alpha}_{j}\left( f_{*}\frac{\rd}{\rd x_{k}} \right) ,\\
    I\left( f_{*}\frac{\rd}{\rd x_{i}} \oplus \widehat{\alpha}_{j} \right) = f_{*}\frac{\rd}{\rd x_{i}} \oplus -\widehat{\alpha}_{j},\\
    \Phi \left( \frac{\rd}{\rd x_{i}} \right) = f_{*} \frac{\rd}{\rd x_{i}} \oplus -\nu_{*} \frac{\rd}{\rd x_{i}} = \left( \frac{\rd}{\rd x_{i}} + \frac{x_{i}^{2}}{2} \frac{\rd}{\rd x_{3}} \right) \oplus x_{i} \widehat{\alpha}_{i}, \\
    \nabla^{+}_{\tfrac{\rd}{\rd x_{i}}}f_{*}  \frac{\rd}{\rd x_{j}} = 0, \qquad \nabla^{-}_{\tfrac{\rd}{\rd x_{i}}}\widehat{\alpha}_{j} = 0
  \end{gather*}
  for $i,j,k,l=1,2$.
\end{exam}

\begin{exam}
  Let $\{f,\xi\}$ be a centroaffine immersion.
  The second fundamental form $h$ coincides with the Ricci tensor field $Ric^{\nabla}$, where $\nabla$ is the induced connection.
  Hence, $Ric^{\nabla}$ is degenerate if and only if $h$ is degenerate.
  Then a centroaffine immersion induces a quasi-Codazzi structure.

  Let $\epsilon$ be a positive number such that $\gamma(x_{1},x_{2}) :=2(x_{1}^{3}+x_{2}^{3})-1<0$ holds for any $(x_{1},x_{2}) \in D_{\epsilon}:=\{ (x_{1},x_{2})\in\R^{2} \;|\; x_{1}^{2}+x_{2}^{2} < \epsilon \}$.
  Set $f:D_{\epsilon}\ni (x_{1},x_{2}) \mapsto (x_{1},x_{2},x_{1}^{3}+x_{2}^{3}+1)\in\R^{3}$.
  Then the pair $\{f,-f\}$ is a centroaffine immersion and
  \begin{align*}
    h\left( \frac{\rd}{\rd x_{i}}, \frac{\rd}{\rd x_{j}} \right) = \delta_{ij} \frac{6x_{i}}{\gamma(x_{1},x_{2})}, \quad \nabla_{\tfrac{\rd}{\rd x_{i}}}\frac{\rd}{\rd x_{j}} = \delta_{ij} \frac{6x_{i}}{\gamma(x_{1},x_{2})} \left(x_{1} \frac{\rd}{\rd x_{1}} + x_{2} \frac{\rd}{\rd x_{2}}\right)
  \end{align*}
  for $i,j = 1,2$.
  Thus, $h$ is degenerate where $x_{1}=0$ or $x_{2}=0$.
  Therefore, the centroaffine immersion induces a quasi-Codazzi structure, not a Codazzi structure.
  We set
  \begin{align*}
    \widehat{\alpha}_{1} = \left. \left( \frac{x_{1}}{\gamma(x_{1},x_{2})^{2}}(-3x_{1} \alpha_{1} -3x_{2}^{2} \alpha_{2} + \alpha_{3}) + \frac{1}{\gamma(x_{1},x_{2})}\alpha_{1} \right) \right|_{f_{*}T\R^{2}}
  \end{align*}
  and
  \begin{align*}
    \widehat{\alpha}_{2} = \left. \left( \frac{x_{2}}{\gamma(x_{1},x_{2})^{2}}(-3x_{1}^{2} \alpha_{1} -3x_{2} \alpha_{2} + \alpha_{3}) + \frac{1}{\gamma(x_{1},x_{2})}\alpha_{2} \right) \right|_{f_{*}T\R^{2}},
  \end{align*}
  where $\alpha_{i} \in \Gamma(T^{*}\R^{3})$ satisfies $\alpha_{i}(\frac{\rd}{\rd x_{j}}) = \delta_{ij}$ for $i,j=1,2,3$.
  The conormal map of $\{f,-f\}$ is obtained by $\nu:\R^{2}\ni (x_{1},x_{2})\mapsto \frac{1}{\gamma(x_{1},x_{2})}(3x_{1}^{2},3x_{2}^{2},-1) \in(\R^{3})^{*}$.
  The induced quasi-Codazzi structure $(h,(E,\theta, I),\Phi,\nabla^{+},\nabla^{-})$ is given by
    
  \begin{gather*}
    E = f_{*}T\R^{2} \oplus (f_{*}T\R^{2})^{*},\\
    \theta \left( f_{*}\frac{\rd}{\rd x_{i}} \oplus \widehat{\alpha}_{j}, f_{*}\frac{\rd}{\rd x_{k}} \oplus \widehat{\alpha}_{l} \right) = \widehat{\alpha}_{l}\left( f_{*}\frac{\rd}{\rd x_{i}} \right) + \widehat{\alpha}_{j}\left( f_{*}\frac{\rd}{\rd x_{k}} \right) ,\\
    I\left( f_{*}\frac{\rd}{\rd x_{i}} \oplus \widehat{\alpha}_{j} \right) = f_{*}\frac{\rd}{\rd x_{i}} \oplus -\widehat{\alpha}_{j},\\
    \Phi \left( \frac{\rd}{\rd x_{i}} \right) = f_{*} \frac{\rd}{\rd x_{i}} \oplus -\nu_{*} \frac{\rd}{\rd x_{i}} =  \frac{\rd}{\rd x_{i}} + 3x_{i}^{2} \frac{\rd}{\rd x_{3}} \oplus 6x_{i}\widehat{\alpha}_{i}\\
    \nabla^{+}_{\tfrac{\rd}{\rd x_{i}}} f_{*}\frac{\rd}{\rd x_{j}} = \delta_{ij} \frac{6x_{i}}{\gamma(x_{1},x_{2})} \left(x_{1} f_{*}\frac{\rd}{\rd x_{1}} + x_{2} f_{*}\frac{\rd}{\rd x_{2}}\right),\\
    \nabla^{-}_{\tfrac{\rd}{\rd x_{i}}} \ha{\alpha}_{j} = \left\{ -\frac{6x_{i}^{2}}{\gamma} \delta_{ij} - \frac{6x_{1}x_{2}}{\gamma} (1-\delta_{ij}) \right\}\ha{\alpha}_{i} - \frac{\rd_{i}\gamma}{\gamma}\ha{\alpha}_{j}
  \end{gather*}
  for $i,j,k,l=1,2$.
\end{exam}

\begin{exam}
  We consider a hypersurface $f:\R^{2}\ni (t,x)\mapsto \gamma(t) + (0,x,x)\in \R^{3}$ and a transversal vector field $\xi := \gamma^{\prime\prime}$, where $\gamma$ is a curve defined by $\gamma(t)=(\cos t, \sin t, t)$.
  The pair $\{f,\xi\}$ defines an affine immersion.
  In particular, the affine immersion is called an {\it affine cylinder} \cite{Nomizu-Sasaki}.
  By straightforward calculations, $\{ f,\xi \}$ is equiaffine and the second fundamental form $h$ is given by
  \begin{gather*}
    h\left( \frac{\rd}{\rd t}, \frac{\rd}{\rd t} \right) =1, \quad h\left( \frac{\rd}{\rd t}, \frac{\rd}{\rd x} \right) = h\left( \frac{\rd}{\rd x}, \frac{\rd}{\rd t} \right) =h\left( \frac{\rd}{\rd x}, \frac{\rd}{\rd x} \right) =0.
  \end{gather*}
  Therefore, $h$ is degenerate for every point.
  We set $\widehat{\alpha}_{1} = (-\sin t\;\alpha_{1} +\cos t\;\alpha_{2})|_{f_{*}T\R^{2}}$ and $\widehat{\alpha}_{2} = (\sin t\;\alpha_{1} -\cos t\;\alpha_{2} + \alpha_{3})|_{f_{*}T\R^{2}}$, where $\alpha_{i} \in \Gamma(T^{*}\R^{3})$ satisfies $\alpha_{i}(\frac{\rd}{\rd x_{j}}) = \delta_{ij}$ for $i,j=1,2,3$.
  The conormal map of $\{f,\xi\}$ is obtained by $\nu:\R^{2}\ni(t,x)\mapsto (-\cos t, -\sin t,0)\in(\R^{3})^{*}$.
  The induced quasi-Hessian structure $(h,(E,\theta,I),\Phi,\nabla^{+},\nabla^{-})$ is given by
  \begin{gather*}
    E = f_{*}T\R^{2} \oplus (f_{*}T\R^{2})^{*},\\
    \theta \left( f_{*}\frac{\rd}{\rd t} \oplus \widehat{\alpha}_{i},\; f_{*}\frac{\rd}{\rd x} \oplus \widehat{\alpha}_{j} \right) = \widehat{\alpha}_{j}\left( f_{*}\frac{\rd}{\rd t} \right) + \widehat{\alpha}_{i}\left( f_{*}\frac{\rd}{\rd x} \right) ,\\
    I\left( f_{*}\frac{\rd}{\rd t}\right) = f_{*}\frac{\rd}{\rd t},\quad I\left( f_{*}\frac{\rd}{\rd x} \right) = f_{*}\frac{\rd}{\rd x},\quad I(\widehat{\alpha}_{i}) = -\widehat{\alpha}_{i}\\
    \Phi \left( \frac{\rd}{\rd t} \right) = \gamma^{\prime} \oplus \alpha_{1}, \quad \Phi \left( \frac{\rd}{\rd x} \right) = (0,0,1) \oplus \mathbf{0},\\
    \nabla^{+}_{\tfrac{\rd}{\rd t}}f_{*} \frac{\rd}{\rd t} = \nabla^{+}_{\tfrac{\rd}{\rd t}} f_{*} \frac{\rd}{\rd x} = \nabla^{+}_{\tfrac{\rd}{\rd x}} f_{*} \frac{\rd}{\rd t} = \nabla^{+}_{\tfrac{\rd}{\rd x}} f_{*} \frac{\rd}{\rd x} = 0,\\
    \nabla^{-}_{\tfrac{\rd}{\rd t}} \widehat{\alpha}_{i} = \nabla^{-}_{\tfrac{\rd}{\rd x}} \widehat{\alpha}_{i} = 0
  \end{gather*}
  for $i,j=1,2$.
\end{exam}

Finally, we describe the relation between a geometric divergence and a quasi-Codazzi structure induced by an equiaffine immersion.

\begin{prop}\upshape
  Let $\{ f,\xi \}$ be an equiaffine immersion and $\nu$ a conormal map of $\{f,\xi\}$.
  A geometric divergence
  \[
  \rho^{G}:M\times M \ni (p,q) \mapsto \langle f(p)-f(q),\nu(q) \rangle \in \R
  \]
  is a weak contrast function of the induced quasi-Codazzi structure of $\{ f,\xi \}$, i.e., $\rho$ satisfies, for $X,Y,Z\in\Gamma(TM)$,
  \begin{align*}
    h(X,Y) &=-\rho^{G}[X|Y], \\
    C(X,Y,Z) &=-\rho^{G}[Z|XY] + \rho[XY|Z],
  \end{align*}
  where $h$ is the second fundamental form and $C$ is defined by \eqref{dfn:cubic tensor of quasi-Codazzi}.
\end{prop}

\proof
By straightforward calculations, we have the following identities:
\begin{align*}
  \rho^{G}[-|-](r) &= \rho^{G}(r,r) = 0;\\
  \rho^{G}[X|-](r) &= X_{p} \langle f(p)-f(q),\nu(q) \rangle \big|_{p=q=r}\\
  &= \langle f_{*}X_{p},\nu(q) \rangle \big|_{p=q=r}\\
  &= 0;\\
  \rho^{G}[-|X](r) &= X_{q}\langle f(p)-f(q),\nu(q) \rangle \big|_{p=q=r}\\
  &= \langle -f_{*}X_{q},\nu(q) \rangle \big|_{p=q=r} + \langle f(p)-f(q),\nu_{*}Y_{q} \rangle \big|_{p=q=r}\\
  &= 0;\\
  \rho^{G}[X|Y](r) &= X_{p}Y_{q} \langle f(p)-f(q),\nu(q) \rangle \big|_{p=q=r}\\
  &= X_{p} \langle -f_{*}Y_{q},\nu(q) \rangle \big|_{p=q=r} + X_{p} \langle f(p)-f(q), \nu_{*}Y_{q} \rangle \big|_{p=q=r}\\
  &= \langle f_{*}X_{p},\nu_{*}Y_{q} \rangle \big|_{p=q=r}\\
  &= -h(X_{r},Y_{r}).
\end{align*}
Moreover, it holds that
\begin{align*}
  \rho^{G}[XY|Z](r)-&\rho^{G}[Z|XY](r)\\
  &= X_{p}\langle f_{*}Y_{p},\nu_{*}Z_{q} \rangle \big|_{p=q=r} - X_{q}\langle f_{*}Z_{p},\nu_{*}Y_{q} \rangle \big|_{p=q=r}\\
  &= \langle \til{\nabla}_{X}f_{*}Y_{p},\nu_{*}Z_{q} \rangle \big|_{p=q=r} - \langle f_{*}Z_{p},\til{\nabla}_{X}\nu_{*}Y_{q} \rangle \big|_{p=q=r}\\
  &= \langle \til{\nabla}_{X}f_{*}Y_{r},\nu_{*}Z_{r} \rangle - \langle f_{*}Z_{r},\til{\nabla}_{X}\nu_{*}Y_{r} \rangle\\
  &=\langle \til{\nabla}_{X}f_{*}Y_{r},\nu_{*}Z_{r} \rangle - X\langle f_{*}Z_{r},\nu_{*}Y_{r} \rangle + \langle \til{\nabla}_{X}f_{*}Z_{r},\nu_{*}Y_{r} \rangle \\
  &= \langle f_{*}(\nabla_{X}Y_{r}) + h(X_{r},Y_{r})\xi_{r}, \nu_{*}Z_{r} \rangle - X\langle f_{*}Z_{r},\nu_{*}Y_{r} \rangle \\ & \qquad \qquad \qquad +\langle f_{*}(\nabla_{X}Z_{r})+ h(X_{r},Z_{r})\xi_{r},\nu_{*}Y_{r}) \rangle \\
  &= \langle f_{*}(\nabla_{X}Y_{r}),\nu_{*}Z_{r} \rangle - X\langle f_{*}Z_{r},\nu_{*}Y_{r} \rangle + \langle f_{*}(\nabla_{X}Z_{r}),\nu_{*}Y_{r} \rangle\\
  &= -h(\nabla_{X}Y_{r},Z_{r}) + Xh(Z_{r},Y_{r}) - h(\nabla_{X}Z_{r},Y_{r})\\
  &= \nabla_{X}h(Y_{r},Z_{r})\\
  &= C(X_{r},Y_{r},Z_{r}),
\end{align*}
where $\til{\nabla}$ is the standard connection on $\R^{n+1}$.
Hence, a geometric divergence is a weak contrast function of the quasi-Codazzi structure.
\qed

\vskip\baselineskip
\vskip\baselineskip

\subsection*{Acknowledgments}
The author would like to thank Profs.~H.~Furuhata and T.~Ohmoto for guiding him to this subject and valuable instructions.
He is also grateful to Prof.~Y.~Numata and Dr.~N.~Nakajima and Dr.~R.~Ueno for their helpful comments during his seminar talks.
This work was supported by JST SPRING, Grant Number JPMJSP2119.

\vskip\baselineskip
\vskip\baselineskip

\section*{Appendix}

\noindent {\sl Proof of Lemma~\ref{lem:Bianchi identity of coherent tangent bundle}} :
Let $\{ \eta_{1},\dots,\eta_{n} \}$ be a frame of $\E$.
For all $i,j\in \{1,\dots,n\}$ and $X,Y\in\Gamma(TM)$, the structure equation $\Omega_{i}^{j}$ of $\nabla^{\E}$ is defined by
\[
R^{\nabla^{\E}}(X,Y)\eta_{i} = \Omega_{i}^{j}(X,Y)\eta_{j}
\]
and satisfies $d\Omega_{i}^{j}= \Omega_{k}^{j} \wedge \omega_{i}^{k} - \omega_{k}^{j} \wedge \Omega_{i}^{k}$ by the second Bianchi identity \cite{Kobayashi-Nomizu}, where $\nabla^{\E}_{X}\eta_{i}=\omega_{i}^{j}(X)\eta_{j}$.
By the assumption, there exists $A\in \Gamma(T^{*}M \otimes \E^{*})$ satisfying
\begin{align*}
  R^{\nabla^{\E}}(X,Y)\eta_{i} &= \frac{1}{n-1} \{ A(Y,\eta_{i})\phi(X) - A(X,\eta_{i})\phi(Y) \}\\
  &=\frac{1}{n-1} \{ A(Y,\eta_{i})\phi_{j}(X) - A(X,\eta_{i})\phi_{j}(Y) \}\eta_{j},
\end{align*}
where $\phi(X)=\phi_{j}(X)\eta_{j}$.
Hence, we have
\[
\Omega_{i}^{j}(X,Y) = \frac{1}{n-1} \{ A(Y,\eta_{i})\phi_{j}(X) - A(X,\eta_{i})\phi_{j}(Y) \}.
\]
For $X,Y,Z\in\Gamma(TM)$, the following equations hold:
\begin{align*}
  d\Omega_{i}^{j} (X,Y,Z) &= \{ X\Omega_{i}^{j}(Y,Z) + Y\Omega_{i}^{j}(Z,X) + Z\Omega_{i}^{j}(X,Y) \\
  &\quad\quad -\Omega_{i}^{j}([X,Y],Z) - \Omega_{i}^{j}([Y,Z],X) - \Omega_{i}^{j}([Z,X],Y) \};
\end{align*}
\begin{align*}
  \Omega_{k}^{j} \wedge \omega_{i}^{k} (X,Y,Z) = \Omega_{k}^{j}(X,Y) \omega_{i}^{k}(Z) + \Omega_{k}^{j}(Y,Z) \omega_{i}^{k}(X) + \Omega_{k}^{j}(Z,X) \omega_{i}^{k}(Y);
\end{align*}
\begin{align*}
  \omega_{k}^{j} \wedge \Omega_{i}^{k} (X,Y,Z) = \omega_{k}^{j}(X)\Omega_{i}^{k}(Y,Z) + \omega_{k}^{j}(Y)\Omega_{i}^{k}(Z,X) + \omega_{k}^{j}(Z)\Omega_{i}^{k}(X,Y);
\end{align*}
\begin{align*}
  (n-1) &\sum_{j=1}^{n} \{ X\Omega_{i}^{j}(Y,Z)\eta_{j} + \Omega_{i}^{j}(Y,Z) \nabla^{\E}_{X}\eta_{j} \}\\
  &= \sum_{j=1}^{n} [ \{ (XA(Z,\eta_{i}))\phi_{j}(Y) + A(Z,\eta_{i})(X\phi_{j}(Y)) - (XA(Y,\eta_{i}))\phi_{j}(Z)\\
  &\qquad\qquad - A(Y,\eta_{i})(X\phi_{j}(Z))\}\eta_{j} + \{ A(Z,\eta_{i})\phi_{j}(Y) - A(Y,\eta_{i})\phi_{j}(Z) \} \nabla^{\E}_{X}\eta_{j} ]\\
  &= XA(Z,\eta_{i})\phi(Y) - XA(Y,\eta_{i})\phi(Z) + A(Z,\eta_{i})\nabla^{\E}_{X}\phi(Y) - A(Y,\eta_{i})\nabla^{\E}_{X}\phi(Z).
\end{align*}
Therefore, it holds that
\begin{align*}
  0&= (n-1) \left[ \sum_{j=1}^{n} \left\{ d\Omega_{i}^{j} - \sum_{k=1}^{n}(\Omega_{k}^{j} \wedge \omega_{i}^{k} + \omega_{k}^{j} \wedge \Omega_{i}^{k}) \right\}(X,Y,Z)\eta_{j} \right] \\
  &= (n-1) \left[ \sum_{j=1}^{n}  \{ X\Omega_{i}^{j}(Y,Z) + Y\Omega_{i}^{j}(Z,X) + Z\Omega_{i}^{j}(X,Y) \right.\\
  &\quad\quad -\Omega_{i}^{j}([X,Y],Z) - \Omega_{i}^{j}([Y,Z],X) - \Omega_{i}^{j}([Z,X],Y) \}\eta_{j}\\
  &\quad\quad\quad\quad - \sum_{j,k=1}^{n} \{ \Omega_{k}^{j}(X,Y) \omega_{i}^{k}(Z) + \Omega_{k}^{j}(Y,Z) \omega_{i}^{k}(X) + \Omega_{k}^{j}(Z,X) \omega_{i}^{k}(Y) \} \eta_{j} \\
  &\qquad\qquad\quad\quad \left. + \sum_{j,k=1}^{n} \{ \omega_{k}^{j}(X)\Omega_{i}^{k}(Y,Z) + \omega_{k}^{j}(Y)\Omega_{i}^{k}(Z,X) + \omega_{k}^{j}(Z)\Omega_{i}^{k}(X,Y) \} \eta_{j} \right] \\
  &= (n-1) \left[ \sum_{j=1}^{n}  \{ X\Omega_{i}^{j}(Y,Z) + Y\Omega_{i}^{j}(Z,X) + Z\Omega_{i}^{j}(X,Y) \} \eta_{j} \right. \\
  &\quad\quad - R^{\nabla^{\E}}([X,Y],Z)\eta_{i} - R^{\nabla^{\E}}([Y,Z],X)\eta_{i} - R^{\nabla^{\E}}([Z,X],Y)\eta_{i}\\
  &\quad\quad\quad\quad - R^{\nabla^{\E}}(X,Y)\nabla^{\E}_{Z}\eta_{i} - R^{\nabla^{\E}}(Y,Z)\nabla^{\E}_{X}\eta_{i} - R^{\nabla^{\E}}(Z,X)\nabla^{\E}_{Y}\eta_{i}\\
  &\qquad\qquad\quad\quad \left. + \sum_{k=1}^{n} \{ \Omega^{k}_{i}(Y,Z) \nabla^{\E}_{X}\eta_{k} + \Omega^{k}_{i}(Z,X) \nabla^{\E}_{Y}\eta_{k} + \Omega^{k}_{i}(X,Y) \nabla^{\E}_{Z}\eta_{k} \} \right]\\
  &= \{ XA(Z,\eta_{i}) \phi(Y) - XA(Y,\eta_{i})\phi(Z) + A(Z,\eta_{i}) \nabla^{\E}_{X}\phi(Y) - A(Y,\eta_{i}) \nabla^{\E}_{X}\phi(Z) \}\\
  &\qquad + \{ YA(X,\eta_{i}) \phi(Z) - YA(Z,\eta_{i})\phi(X) + A(X,\eta_{i}) \nabla^{\E}_{Y}\phi(Z) - A(Z,\eta_{i}) \nabla^{\E}_{Y}\phi(X) \}\\
  &\qquad\qquad + \{ ZA(Y,\eta_{i}) \phi(X) - ZA(X,\eta_{i})\phi(Y) + A(Y,\eta_{i}) \nabla^{\E}_{Z}\phi(X) - A(X,\eta_{i}) \nabla^{\E}_{Z}\phi(Y) \}\\
  &\qquad\qquad\qquad - \{ A(Z,\eta_{i})\phi([X,Y]) - A([X,Y],\eta_{i})\phi(Z) \}\\
  &\qquad\qquad\qquad\qquad - \{ A(X,\eta_{i})\phi([Y,Z]) - A([Y,Z],\eta_{i})\phi(X) \}\\
  &\qquad\qquad\qquad\qquad\qquad  - \{ A(Y,\eta_{i})\phi([Z,X]) - A([Z,X],\eta_{i})\phi(Y) \}\\
  &\qquad\qquad\qquad\qquad\qquad\qquad - \{ A(Y,\nabla^{\E}_{Z}\eta_{i})\phi(X) - A(X,\nabla^{\E}_{Z}\eta_{i})\phi(Y) \}\\
  &\qquad\qquad\qquad\qquad\qquad\qquad\qquad - \{ A(Z,\nabla^{\E}_{X}\eta_{i})\phi(Y) - A(Y,\nabla^{\E}_{X}\eta_{i})\phi(Z) \}\\
  &\qquad\qquad\qquad\qquad\qquad\qquad\qquad\qquad - \{ A(X,\nabla^{\E}_{Y}\eta_{i})\phi(Z) - A(Z,\nabla^{\E}_{Y}\eta_{i})\phi(X) \}\\
  &= -\{ XA(Y,\eta_{i}) - YA(X,\eta_{i})- A([X,Y],\eta_{i}) - A(Y,\nabla^{\E}_{X}\eta_{i}) + A(X,\nabla^{\E}_{Y}\eta_{i}) \}\phi(Z)\\
  &\qquad -\{ YA(Z,\eta_{i}) - ZA(Y,\eta_{i})- A([Y,Z],\eta_{i}) - A(Z,\nabla^{\E}_{Y}\eta_{i}) + A(Y,\nabla^{\E}_{Z}\eta_{i}) \}\phi(X)\\
  &\qquad\qquad -\{ ZA(X,\eta_{i}) - XA(Z,\eta_{i})- A([Z,X],\eta_{i}) - A(X,\nabla^{\E}_{Z}\eta_{i}) + A(Z,\nabla^{\E}_{X}\eta_{i}) \}\phi(Y).\\
\end{align*}
For any $X,Y\in\Gamma(TM)$, if $\rank \phi_{p} \geq 3$ at $p\in M$, then there exists $Z_{p}\in T_{p}M$ such that $\phi_{p}(X_{p})$, $\phi_{p}(Y_{p})$, and $\phi_{p}(Z_{p})$ are linearly independent and $\phi_{p}(Z_{p})\neq 0$.
By the above equation, we have \eqref{eq:mainthm1_condition1} at $p$.
Since $\rank \phi \geq 3$ almost everywhere, \eqref{eq:mainthm1_condition1} holds on $M$.
\qed

\vskip\baselineskip
\vskip\baselineskip


\begin{thebibliography}{99999999}

%
\bibitem{Amari}
S. Amari,
{\it Information Geometry and Its Application},
Applied Math. Sci., 194, Springer (2016).
%
\bibitem{Amari-Nagaoka}
S. Amari, H. Nagaoka,
{\it Method of Information Geometry},
A.M.S., Oxford Univ. Press (2000).
%
\bibitem{Belgun}
F. Belgun,
Projective and conformal flatness,
Lecture note at Hamburg University,
https://www.math.uni-hamburg.de/home/belgun/Weyl-proj.pdf.
%
%
%
%
%
\bibitem{Dajczer-Tojeiro}
M. Dajczer, R. Tojeiro,
\textit{Submanifold Theory: Beyond an Introduction},
Springer New York (2019).
%
%
\bibitem{DNV}
F. Dillen, K. Nomizu, L. Vranken,
Conjugate Connections and Radon's Theorem in Affine Differential Geometry,
Monatsh. Math. {\bf 109} (1990), 221--235.
%
\bibitem{Eguchi}
S. Eguchi,
Geometry of minimum contrast,
Hiroshima Math. J. {\bf 22} (1992), 631--647.
%
\bibitem{Eisenhart}
L. P. Eisenhart,
Non-Riemannian Geometry,
Amer. Math. Soc. Colloq. Publ. {\bf 8} (1927).
%
%
\bibitem{Kayo}
K. Kayo,
Statistical manifold with degenerate metric,
Information Geometry {\bf 7} (2024), 229--252.
%
\bibitem{Kobayashi-Nomizu}
S. Kobayashi, K. Nomizu,
\textit{Foundations of Differential Geometry Volume I},
Interscience Tracts in Pure and Applied Math. 15, Interscience Publishers (1963).
%
%
\bibitem{Kurose}
T. Kurose,
On the divergences of 1-conformally flat statistical manifolds,
T\^{o}hoku Math. J., {\bf 46} (1994), 427--433.
%
\bibitem{Lauritzen}
S. L. Lauritzen,
{\em Statistical manifolds}, Differential Geometry in Statistical Inferences,
Institute of Mathematical Statistics Lecture Notes - Monograph Series {\bf 10} (1987), 163--216.
%
\bibitem{LS05}
M.-A. Lawn, L. Sch\"{a}fer,
Decompositions of para-complex vector bundles and para-complex affine immersions,
Result. Math. {\bf 48} (2005), 246--274.
%
\bibitem{Matsuzoe}
H. Matsuzoe,
Statistical manifolds and affine differential geometry,
Advanced Studies in Pure Mathematics {\bf 57} (2010), 303--321.
%
\bibitem{Matumoto}
T. Matumoto,
Any statistical manifold has a contrast function -- On the $C^3$-functions taking the minimum at the diagonal of the product manifold,
Hiroshima Math. J. {\bf 23} (1993), 327--332.
%
\bibitem{Nakajima-Ohmoto}
N. Nakajima, T. Ohmoto,
The dually flat structure for singular models,
Information Geometry. J. {\bf 4} (2021), 31--64.
%
\bibitem{Nomizu-Sasaki}
K. Nomizu, T. Sasaki,
{\it Affine Differential Geometry: Geometry of Affine Immersions},
Cambridge Univ. Press (1994).
%
\bibitem{Radon}
J. Radon,
Die Grundgleichungen der affinen Fl\"{a}chentheorie,
Leipziger Berichte {\bf 70} (1918), 91--107.
%
%
\bibitem{SUY-F1}
K. Saji, M. Umehara, K. Yamada,
Coherent tangent bundles and Gauss-Bonnet formulas for wave fronts,
J. Geom. Anal. {\bf 22} (2012), 383--409.
%
%
\bibitem{Shima}
H. Shima,
\textit{The geometry of Hessian structure},
World Scientific (2007).
%
%
\bibitem{Weyl}
H. Weyl,
Zur Infinitesimalgeometrie: Einordnung der projektiven und der konformen Auffassung,
Göttinger Nachrichten (1921), 99--112. 

\end{thebibliography}
\end{document}